\documentclass{amsart}
\usepackage{amsmath,amsthm,amssymb}
\usepackage{graphicx}

\newtheorem{theorem}{Theorem}[section]
\newtheorem{proposition}[theorem]{Proposition}
\newtheorem{lemma}[theorem]{Lemma}
\newtheorem{corollary}[theorem]{Corollary}

\theoremstyle{definition}
\newtheorem{definition}[theorem]{Definition}
\newtheorem{remark}[theorem]{Remark}

\newtheorem{step}{Step}

\title[Cork twisting exotic Stein 4-manifolds]{Cork twisting exotic Stein 4-manifolds}
 
\author[AKBULUT and YASUI]{Selman Akbulut and Kouichi Yasui}
\thanks{The first named author is partially supported by NSF grants DMS 0905917, DMS 1065955, and the second named author was partially supported by GCOE, Kyoto University, by KAKENHI~21840033 and by JSPS Research Fellowships for Young Scientists.}
\date{February 14, 2011. \: \textit{Revised}: July 20, 2011}
\subjclass[2000]{Primary~57R55, Secondary~57R65}
\keywords{4-manifold; cork; Stein surface; contact manifold; exotic embedding}
\address{Department~of~Mathematics, Michigan State University, E.Lansing, MI, 48824, USA}
\email{akbulut@math.msu.edu}
\address{Department~of~Mathematics, Graduate School~of~Science, Hiroshima~University, 1-3-1 Kagamiyama, Higashi-Hiroshima, Hiroshima, 739-8526, Japan}\email{kyasui@hiroshima-u.ac.jp}

\begin{document}

\begin{abstract}From any $4$-dimensional oriented handlebody  $X$ without 3- and 4-handles and with $b_2\geq 1$, we construct arbitrary many compact Stein $4$-manifolds which are mutually homeomorphic but not diffeomorphic to each other, so that their topological invariants (their fundamental groups, homology groups, boundary homology groups, and  intersection forms) coincide with those of $X$. 
We also discuss the induced contact structures on their boundaries. 
Furthermore, for any smooth $4$-manifold pair $(Z,Y)$ such that the complement $Z-\textnormal{int}\,Y$ is a handlebody without 3- and 4-handles and with $b_2\geq 1$, we construct arbitrary many exotic embeddings of a compact $4$-manifold $Y'$ into $Z$, such that $Y'$ has the same topological invariants as $Y$. 
\end{abstract}

\maketitle

\vspace{-.3in}

\section{Introduction}\label{sec:intro} A basic problem of $4$-manifold topology is to find all exotic copies of  smooth $4$-manifolds, in particular to find various methods of constructing different smooth structures on $4$-manifolds (e.g.\  logarithmic transform \cite{D}, Fintushel-Stern's rational blowdown~\cite{FS1} and knot surgery~\cite{FS2}). The purpose of this paper is to approach this problem by corks and give applications. Since different smooth structures on a $4$-manifold can be explained by existence corks which divide the manifold into two Stein pieces \cite{AM2},  cork twisting Stein manifolds is a central theme of this paper. \smallskip

The first cork was introduced in~\cite{A1}, and was used in \cite{A2} to construct a pair of two simply connected compact $4$-manifolds with boundary and second betti number  $b_2=1$ which are homeomorphic but non-diffeomorphic. Later it turned out that cork twists easily give many such pairs (Akbulut-Matveyev~\cite{AM1}, the authors~\cite{AY1}), where each pair consists of a Stein $4$-manifold and a non-Stein $4$-manifold, hence they are not diffeomorphic.\smallskip

It is thus interesting to find exotic Stein $4$-manifold pairs.  Uniqueness of diffeomorphism types of Stein 4-manifolds bounding certain $3$-maniolds are known $($e.g.\ $\#_n\,S^1\times S^2$ $(n\geq 0)$, for more examples see~\cite{We} and the references mentioned therein$)$. By contrast, Akhmedov-Etnyre-Mark-Smith~\cite{AEMS} constructed infinitely many simply connected compact Stein $4$-manifolds which are mutually homeomorphic but non-diffeomorphic, using knot surgery. Moreover, the induced contact structures on their boundary are mutually isomorphic. Though these 4-manifolds have large second betti number, later  in ~\cite{AY2} for each $b_2\geq 1$, by using corks, the authors constructed pairs of simply connected compact Stein $4$-manifolds which are homeomorphic but non-diffeomorphic.\smallskip

In this paper, by using properties of Stein $4$-manifolds we extend the previous simple cork constructions to an explicit algorithm. Here ($4$-dimentional oriented) \textit{$2$-handlebody} means a compact, connected, oriented smooth $4$-manifold obtained from the $4$-ball by attaching $1$- and $2$-handles. The algorithm goes roughly as follows: Take any $2$-handlebody with $b_2\geq 1$, then change the handle diagram into a certain form and add appropriate corks to  produce compact Stein 4-manifolds; then by twisting these corks detect the change of smooth structures by the adjunction inequalities. This construction generalizes the carving technique of \cite{AM2}.
\smallskip

This process here gives arbitrary many mutually homeomorphic but not diffeomorphic compact Stein 4-manifolds which have the same topological invariants as the given 2-handlebody (see Theorems~\ref{sec:algorithm1:main theorem} and~\ref{sec:variant:main theorem}, for details). We obtain: 
\begin{theorem}\label{sec:intro:th:main}Let $X$ be any $4$-dimentional $2$-handlebody with the second betti number $b_2(X)\geq 1$. Then, for each $n\geq 1$, there exist $2$-handlebodies $X_i$ $(0\leq i\leq n)$ with the following properties: \smallskip \\
$(1)$ The fundamental group, the integral homology groups, the integral homology groups of the boundary, and the intersection form of each $X_i$ $(0\leq i\leq n)$ are isomorphic to those of $X$. \smallskip \\
$(2)$ $X_i$ $(0\leq i\leq n)$ are mutually homeomorphic but non-diffeomorphic. \smallskip \\
$(3)$ Each $X_i$ $(1\leq i\leq n)$  has a Stein structure. \smallskip \\
$(4)$ $X$ can be embedded into $X_0$. Hence, $X_0$ does not admit any Stein structure if $X$ cannot be embedded into any simply connected minimal symplectic $4$-manifold with $b^+_2>1$. $($For more non-existence conditions see Theorems~\ref{sec:algorithm1:main theorem} and~\ref{sec:variant:main theorem}.$)$\smallskip \\
$(5)$ Each $X_i$ $(0\leq i\leq n)$ can be embedded into $X$. 
\end{theorem} 

As far as the authors know, this result is new even when we ignore Stein structures. Actually, this theorem gives exotic smooth structures for a large class of compact 4-manifolds with boundary (see also Corollary~\ref{sec:non-stein:cor}). In Section~\ref{sec:non-stein}, we also construct arbitrary many exotic non-Stein 4-manifolds.\smallskip

For a given embedding of a 4-manifold, applying the algorithm to its complement, we obtain arbitrary many exotic embeddings of a 4-manifold which has the same the topological invariants as the given manifold (see Theorems~\ref{sec:algorithm1:exotic knottings} and~\ref{sec:variant:exotic knottings}).
\begin{theorem}\label{sec:into:th:exotic embedding}
Let $Z$ and $Y$ be compact connected oriented smooth $4$-manifolds $($possibly with boundary$)$. Suppose that $Y$ is embedded into $Z$ and that its complement $X:=Z-\textnormal{int}\,Y$ is a $2$-handlebody with $b_2(X)\geq 1$. Then, for each $n\geq 1$, there exist mutually diffeomorphic compact connected oriented smooth $4$-manifolds $Y_i$ $(0\leq i\leq n)$ embedded into $Z$ with the following properties. \smallskip \\
$(1)$ The pairs $(Z,Y_i)$ $(0\leq i\leq n)$ are mutually homeomorphic but non-diffeomorphic.\smallskip \\
$(2)$ The fundamental group, the integral homology groups, the integral homology groups of the boundary, and the intersection form of $Y_i$'s $(0\leq i\leq n)$ are isomorphic to those of $Y$.  \smallskip \\
$(3)$ The each complement $X_i:=Z-\textnormal{int}\,Y_i$ $(0\leq i\leq n)$ has the properties of the $X_i$  in Theorem~\ref{sec:intro:th:main} above (corresponding to $X$). 
\end{theorem}
Note that any compact connected oriented smooth $4$-manifold $Z$ (possibly with boundary) has such a submanifold $Y$, because the 4-ball contains a $2$-handlebody $S^2\times D^2$, for example. Hence this theorem shows that every compact connected oriented smooth $4$-manifold has arbitrary many exotic embeddings into it, and has arbitrary many compact sub 4-manifolds which are mutually homeomorphic but not diffeomorphic. \smallskip

We further state the properties of $X_i$'s in Theorem~\ref{sec:intro:th:main}. At least in the case of $b_2(X)=1$, the induced contact structures on the boundary $\partial X_i$'s have the property below. This also shows that $X_i$'s are mutually non-diffeomorphic. 

\begin{corollary}\label{sec:intro:cor:contact}Let $X$ be any $2$-handlebody with $b_2(X)=1$. Suppose that the intersection form of $X$ is non-zero. Fix $n\geq 1$ and denote by $X_i$ $(1\leq i\leq n)$ the corresponding compact Stein $4$-manifold in Theorem~\ref{sec:intro:th:main}. Let $\xi_i$ $(1\leq i\leq n)$ be the contact structure on the boundary $\partial X_i \; (\cong \partial X_1)$ induced by the Stein structure on $X_i$.  Then the each smooth $4$-manifold $X_i$ $(1\leq i\leq n-1)$ admits no Stein structure compatible with $\xi_j$ for any $j>i$. 
\end{corollary}

It is interesting to discuss cork structures of 4-manifolds (see~\cite{AY1}, \cite{AY3}, \cite{AY4}).  In~\cite{AY4}, the authors  constructed the following example: For each $n\geq 2$, there are $n$ mutually disjoint embeddings of the same cork into a simply connected compact $4$-manifold $Z_n$ with boundary, so that twisting $Z_n$ along each copy of the cork produces mutually distinct $n$ smooth structures on  $Z_n$.
However, $b_2(Z_n)$ increases when $n$ increases. The above $X_i$'s have the structures below.
We also discuss infinitely many disjoint embeddings in Section~\ref{sec:noncompact}.
\begin{corollary}\label{sec:main:cor:cork:compact case}Let $X$ be any $2$-handlebody with $b_2\geq 1$. For each $n\geq 1$, there exist $2$-handlebodies $X_i$ $(0\leq i\leq n)$, a cork $(C,\tau)$, disjointly embedded copies $C_i$ $(1\leq i\leq n)$ of $C$ into $X_0$ with the following properties:\smallskip\\
$(1)$ $X_i$ $(0\leq i\leq n)$ are mutually homeomorphic but non-diffeomorphic. \smallskip\\
$(2)$ Each $X_i$ $(1\leq i\leq n)$ is the cork twist of $X_0$ along $(C_i,\tau)$.\smallskip\\
$(3)$ Each $X_i$ $(1\leq i\leq n)$ is the manifold of  Theorem~\ref{sec:intro:th:main}, corresponding to this $X$.
\end{corollary}
In a forth coming paper, we will discuss Theorems~\ref{sec:intro:th:main} and~\ref{sec:into:th:exotic embedding} in the case of  $b_2(X)=0$,  under some conditions. \smallskip

This paper is organized as follows. In Sections~\ref{sec:cork} and \ref{sec:stein_contact} we briefly discuss basics of corks, Stein 4-manifolds, and contact 3-manifolds. In Section~\ref{sec:modification}, we study effects of certain operations related to corks. In Section~\ref{sec:main_algorithm}, we give the algorithm and prove Theorems~\ref{sec:intro:th:main} and \ref{sec:into:th:exotic embedding} and Corollary~\ref{sec:main:cor:cork:compact case}. In Section~\ref{sec:variant}, we strengthen the algorithm. In Section~\ref{sec:contact}, we prove Corollary~\ref{sec:intro:cor:contact}. In Section~\ref{sec:noncompact}, we construct infinitely many disjoint embeddings of a fixed cork into a noncompact $4$-manifolds. In Section~\ref{sec:example}, we apply Theorems~\ref{sec:intro:th:main} and \ref{sec:into:th:exotic embedding} to some examples $X$ and $(Y,Z)=(S^4, \Sigma_g\times D^2)$ $(g\geq 1)$, where $\Sigma_g$ denotes the closed surface of genus $g$. In Section~\ref{sec:non-stein}, we construct arbitrary many compact Stein 4-manifolds and arbitrary many non-Stein 4-manifolds which are mutually homeomorphic but non-diffeomorphic. 
\medskip\\
\textbf{Acknowledgements.} The authors would like to thank the referee for helpful comments and pointing out minor errors. The second author also would like to thank Kazunori Kikuchi, Takefumi Nosaka and Motoo Tange for useful comments. 
\section{Corks}\label{sec:cork}
In this section, we recall corks. For more details, the reader can consult ~\cite{AY1}.

\begin{definition}
Let $C$ be a compact contractible Stein $4$-manifold with boundary and $\tau: \partial C\to \partial C$ an involution on the boundary. 
We call $(C, \tau)$ a \textit{cork} if $\tau$ extends to a self-homeomorphism of $C$, but cannot extend to any self-diffeomorphism of $C$. 
For a cork $(C,\tau)$ and a smooth $4$-manifold $X$ which contains $C$, a {\it cork twist of $X$ along $(C,\tau)$ }is defined to be the smooth $4$-manifold obtained from $X$ by removing the submanifold $C$ and regluing it via the involution $\tau$.  Note that, any cork twist does not change the homeomorphism type of $X$ (see the remark below). A cork $(C, \tau)$ is called a \textit{cork of $X$} if the cork twist of $X$ along $(C, \tau)$ is not diffeomorphic to $X$. 
\end{definition}
\begin{remark}
In this paper, we always assume that corks are contractible. (We did not assume this in the more general definition of~\cite{AY1}.) Freedman's theorem (cf.\ \cite{B}) implies that every self-diffeomorphism of the boundary $\partial C$ extends to a self-homeomorphism of $C$, when $C$ is a compact contractible smooth $4$-manifold. 
\end{remark}

\begin{definition}Let $W_n$ be the contractible smooth 4-manifold shown in Figure~$\ref{fig1}$. Let $f_n:\partial W_n\to \partial W_n$ be the obvious involution obtained by first surgering $S^1\times D^3$ to $D^2\times S^2$ in the interior of $W_n$, then surgering the other imbedded $D^2\times S^2$ back to $S^1\times D^3$ (i.e. replacing the dot and ``0'' in Figure ~$\ref{fig1}$). Note that the diagram of $W_n$ is induced from a symmetric link.
\begin{figure}[ht!]
\begin{center}
\includegraphics[width=1.1in]{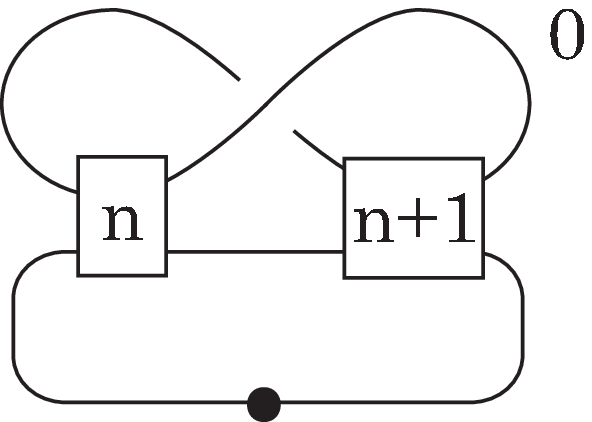}
\caption{$W_n$}
\label{fig1}
\end{center}
\end{figure}
\end{definition}


\begin{theorem}[{\cite[Theorem 2.5]{AY1}}]\label{th:cork}
For $n\geq 1$, the pair $(W_n, f_n)$ is a cork. 
\end{theorem}
\section{Stein $4$-manifolds and contact $3$-manifolds}\label{sec:stein_contact}In this section, we briefly recall basics of Stein $4$-manifolds and contact $3$-manifolds. For the definition of basic terms and more details, the reader can consult Gompf-Stipsicz~\cite{GS} and Ozbagci-Stipsicz~\cite{OS1}.  
In this paper, we use Seifert framings and abbreviate them to framings. (When a knot goes over 4-dimentional $1$-handles, then convert the diagram into the dotted circle notation and calculate its Seifert framing. cf.~\cite{GS}). We use the following terminologies throughout this paper. 

\begin{definition}\label{sec:stein:def:handlebody}
$(1)$ For a Legendrian knot $K$ in $\#n(S^1\times S^2)$ $(n\geq 0)$, we denote by $tb(K)$ and $r(K)$ the Thurston-Bennequin number and the rotation number of $K$, respectively.\smallskip\\
$(2)$ We call a compact connected oriented 4-dimentional handlebody a \textit{$2$-handlebody} if it consists of one $0$-handle and $1$- and $2$-handles. We call a subhandlebody a \textit{sub $1$-handlebody} if it consists of $0$- and $1$-handles of the whole handlebody.\smallskip\\
$(3)$ We call a $2$-handlebody a \textit{Legendrian handlebody} if its $2$-handles are attached to an oriented framed Legendrian link in $\partial(D^4\cup \textit{$1$-handles})=\#n(S^1\times S^2)$ $(n\geq 0)$. It is known that every $2$-handlebody can be changed into a Legendrian handlebody by an isotopy of the attaching link of $2$-handles, and orienting its components. \smallskip\\
$(4)$ We call a Legendrian handlebody a \textit{Stein handlebody} if the framing of its each $2$-handle $K$ is $tb(K)-1$. 
\end{definition}


Next we recall the following useful theorem. 

\begin{theorem}[Eliashberg~\cite{E1}, cf.~\cite{G}, \cite{GS}]\label{sec:stein:th:stein_existence}
A compact, connected, oriented, smooth $4$-manifold admits a Stein structure if and only if it can be represented as a Stein handlebody. 
\end{theorem}
We call a compact smooth $4$-manifold with a Stein structure a \textit{compact Stein $4$-manifold}. Recall that a Stein structure induces an almost complex structure. Thus the first Chern class $c_1$ of a compact Stein $4$-manifold is defined. The following useful theorems are known and play important roles in this paper. 
\begin{theorem}[Gompf~\cite{G}, cf.\ \cite{GS}]\label{sec:stein:th:stein_rotation}
Let $X$ be a Stein handlebody. The first Chern class $c_1(X)\in H^2(X;\mathbf{Z})$ is represented by a cocycle whose value on each $2$-handle $h$ attached along a Legendrian knot $K$ is $r(K)$. Here each $2$-handle is oriented according to the orientation of the corresponding Legendrian knot.\end{theorem}

Note that the theorem below contains the case where the genus and the self-intersection number are zero, unlike the usual adjunction inequality for closed 4-manifolds. 
\begin{theorem}[Akbulut-Matveyev~\cite{AM1}, cf.~\cite{OS1}]Let $X$ be a compact Stein $4$-manifold and $\Sigma$ a smoothly embedded genus $g\geq 0$ closed surface in $X$. Denote by $[\Sigma]$ the second homology class of $X$ represented by $\Sigma$. If $[\Sigma]\neq 0$, then the following adjunction inequality holds: 
\begin{equation*}
[\Sigma]^2+\lvert \langle c_1(X), [\Sigma] \rangle\rvert\leq 2g-2.
\end{equation*}
\end{theorem}
\begin{proof}For the completeness, we give a minor correction to the proof of \cite[Theorem~13.3.8]{OS1}. In the $g=0$ case, apply the same argument as the $g\geq 1$ case (Since \cite[Theorem~13.3.6]{OS1} also holds in the $g=0$ case (\cite{FS3}), one can apply.).
\end{proof}
We also use the following lemma, which is easily checked by Figure~\ref{fig2}. 
\begin{lemma}\label{sec:stein:lem:zigzag}Let $K$ be a Legendrian knot in $\#n (S^1\times S^2)$ $(n\geq 0)$. For any integer pair $(t,d)$ with $t\geq 1$ and $0\leq d\leq t$, by locally adding zig-zags to $K$ upward or downward, $K$ can be changed so that the following \textnormal{(i)} and \textnormal{(ii)} are satisfied.\smallskip 
\begin{itemize}
 \item [(i)] The Thurston-Bennequin number of $K$ decreases by $t$.\smallskip 
 \item [(ii)] The rotation number of $K$ increases by $2d-t$. 
\end{itemize}
\begin{figure}[ht!]
\begin{center}
\includegraphics[width=1.6in]{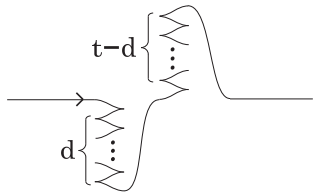}
\caption{adding zig-zags}
\label{fig2}
\end{center}
\end{figure}
\end{lemma}

Compact Stein $4$-manifolds are known to admit useful embeddings. 
\begin{theorem}[Lisca-Mati\'{c}~\cite{LM1}, Akbulut-Ozbagci~\cite{AO3}]\label{sec:stein:th:closing of Stein}
Every compact Stein $4$-manifold can be embedded into a minimal closed complex surface of general type with $b_2^+>1$; and can be embedded into a simply connected, minimal, closed, symplectic $4$-manifold with $b_2^+>1$. Here minimal means that there is no smoothly embedded $2$-sphere with the self-intersection number $-1$. 
\end{theorem}
\begin{proof}(simply connectedness) Since ``simply connected'' is not claimed in \cite{LM1} and \cite{AO3}, we explain this part for completeness. We follow the proof in~\cite{AO3}. We first attach 2-handles to a given Stein 4-manifold to make it simply connected Stein 4-manifold, then apply the procedure prescribed in~\cite{AO3}. Since this results attaching 2-, 3- and 4-handles to the boundary, the simply connectedness is preserved.
\end{proof}

A compact Stein $4$-manifold $X$ induces a contact structure $\xi$ on its boundary $\partial X$. If its Chern class $c_1(\xi)\in H^2(\partial X;\mathbf{Z})$ is a torsion, then the contact invariant $d_3(\xi)\in \mathbf{Q}$ (called the 3-dimentional invariant) is defined by 
\begin{equation*}
d_3(\xi)=\frac{1}{4}(c_1(X)^2-2e(X)-3\sigma(X)), 
\end{equation*}
where $e(X)$ and $\sigma(X)$ denotes the Euler characteristic and the signature of $X$, respectively. For a computation of $c_1(X)^2$, see \cite{GS} and \cite{OS1}. The lemma below is easily verified. 
\begin{lemma}[cf.\cite{GS}]\label{sec:stein:lem:contact}Let $X$ be a compact Stein $4$-manifold with $b_2(X)=1$. Denote the generator of the second homology group of $X$ by $v$. Suppose $v^2\neq 0$, then
\begin{equation*}
c_1(X)^2=\dfrac{{\langle c_1(X), v \rangle}^2}{v^2}.
\end{equation*}
\end{lemma}
\section{$W^+(p)$- and $W^-(p)$-modifications}\label{sec:modification}
In this section, we study the effects of the operations below. We first define them for smooth 2-handlebodies and later redefine them for Legendrian handlebodies. In this paper, the words the ``attaching circle of a $2$-handle'' and a ``smoothly embedded surface'' are often abbreviated to  a ``$2$-handle'' and a ``surface'', if they are clear from the context. 
\begin{definition}Assume $p\geq 1$. Let $K$ be a 2-handle of a (smooth) 2-handlebody. Take a small segment of the attaching circle of $K$ as in the first row of Figure~\ref{fig3}.

We call the local operations shown in the left and the right side of Figure~\ref{fig3} a \textit{$W_1^+(p)$-modification} to $K$ and a \textit{$W_1^-(p)$-modification} to $K$, respectively. Here we do not change the framing of $K$ (ignore the orientations shown in the figure). They are clearly related by a cork twist along $(W_1,f_1)$ as shown in the figure.
\begin{figure}[ht!]
\begin{center}
\includegraphics[width=4.5in]{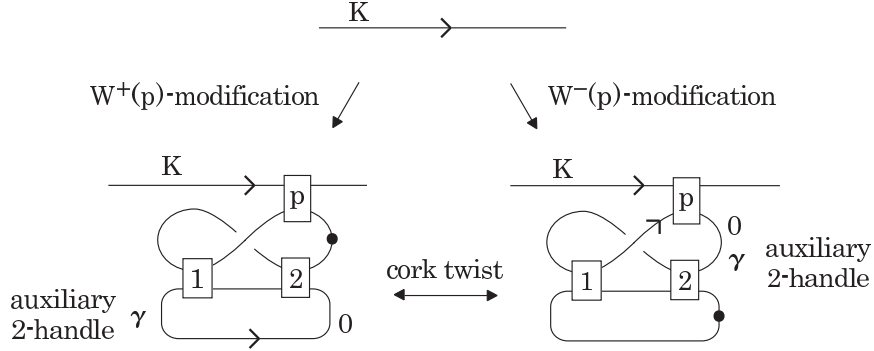}
\caption{$W_1^{\pm}(p)$-modifications $(p\geq 1)$  (the framing of $K$ is unchanged)}
\label{fig3}
\end{center}
\end{figure}

 We will call the $0$-framed $2$-handle $\gamma$ on the left (or  right) side of the Figure~\ref{fig3} the \textit{auxiliary $2$-handle} of the $W_1^{\pm}(p)$-modification of $K$. We will use the same symbol $K$ for the new $2$-handle obtained from the original $K$ of $X$ by the modification.

For convenience, we refer the $W_1^+(0)$- and $W_1^-(0)$-modifications as undone operations. 
For brevity, sometimes we will call these operations  $W_1^+$- and $W_1^-$-modifications when we do not need to specify the coefficients, or call them as $W_1$-modifications when we do not need to specify both the coefficient and $\pm$. Clearly the name of this operation comes from the $W_1$ cork of \cite{AY1}. Similarly we can talk about  \textit{$W^{\pm}(p)$-modification} for any cork $(W,f)$ coming from a symmetric link. 
\end{definition}

For the rest of this paper we will discuss the effects of $W$-modification where $(W,f)=(W_1,f_1)$. In the rest of this section, we assume $p\geq 1$.

\begin{proposition}\label{prop:relation of attachments} Let $K$ be a $2$-handle of a $2$-handlebody $X$. Any $W$-modification to $K$ do not change the isomorphism classes of the fundamental group, the integral homology groups, the integral homology groups of the bounadry $\partial X$, and the intersection form of $X$. 
\end{proposition}
\begin{proof}
Since the $0$-framed auxiliary 2-handle links with the 1-handle algebraically once, each operation does not change the fundamental group, the integral homology groups, and the intersection form. We next check the boundary. Recall that the integral homology groups of the bounadry of any simply connected $2$-handlebody are determined by its intersection form (cf.~\cite{GS}). So we first replace the dots of the dotted circles of $X$ with $0$'s, that is, surgery $S^1\times D^3$'s to $D^2\times S^2$'s. We now have a simply connected 2-handlebody. Next apply the $W$-modification to $K$. This modification keeps the intersection form and the simply connectedness. Moreover, the boundary of this result is diffeomorphic to the boundary of the result of the $W$-modification to $K$ of $X$. Therefore any $W$-modification do not affect the homology groups of the boundary $\partial X$. 
\end{proof}
\begin{proposition}\label{sec:attachments:prop:genus of attachments}
Apply a $W^+(p)$-modification to a $2$-handle $K$ of a $2$-handlebody $X$. Let $X^+$ and $\gamma$ denote the result of $X$ and the auxiliary $2$-handle, respectively. Suppose that the attaching circle of the original $K$ of $X$ spans a smoothly embedded genus $g$ surface in a sub 1-handlebody $\natural_n (S^1\times D^3)$ $(n\geq 0)$ of $X$. Then the new $K$ of $X^+$ spans a smoothly embedded genus $g+p$ surface in a sub 1-handlebody of $X^+$ after sliding over the $2$-handle $\gamma$ $p$-times $($homologically, this changes $K$ to $K-p\gamma$$)$.
\end{proposition}
\begin{proof}The new $K$ is obtained by a band summing  the original $K$ and the knot $U$ in the first picture of Figure~\ref{fig4}. Hence it suffices to check that $U$ spans a smoothly embedded  surface of genus $p$ after sliding over the $2$-handle $\gamma$ $p$-times. 
Introduce a canceling 1- and 2-handle pair and slide $\gamma$ (geometrically) twice, 
then we get the second picture. Isotopy gives the third picture. We then slide the
knot $U$ over the 0-framed unknot $p$-times so that $U$ does not link with the lower
dotted circle. We get the fourth picture, by ignoring two 2-handles, and isotopy. We
can now easily see that $U$ bounds a  surface of genus $p$ by the standard argument (cf.\
\cite[Exercise.4.5.12.(b)]{GS}). One can check that $U$ is the boundary of $D^2$ with $2p$ bands attached. Note that in the beginning, we slided $\gamma$ over the $-1$ framed 2-handle, which does not affect the result because the sliding was over the canceling 2-handle algebraically zero times. 
%
%
\end{proof}
\begin{figure}[ht!]
\begin{center}
\includegraphics[width=4.0in]{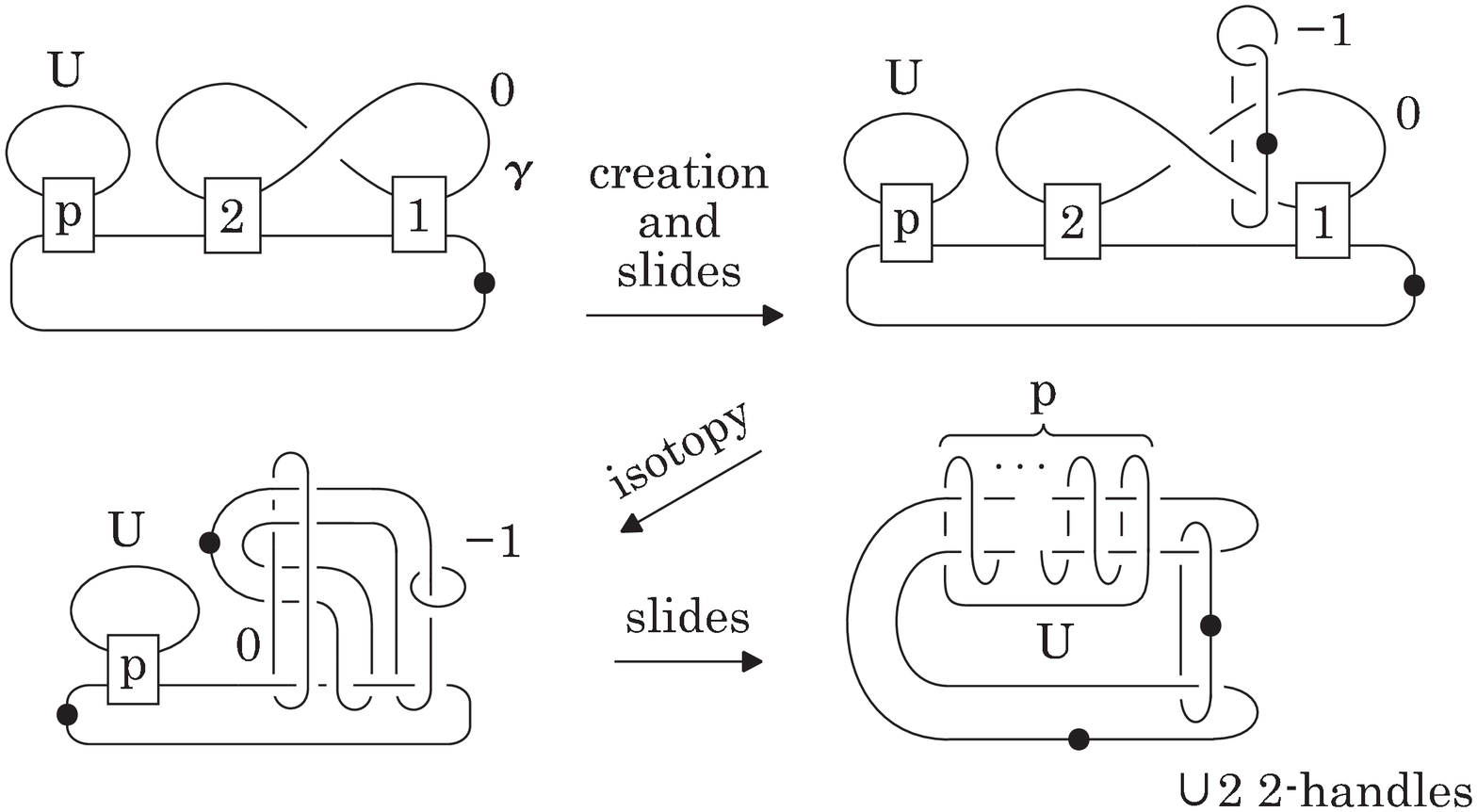}
\caption{}
\label{fig4}
\end{center}
\end{figure}


\begin{lemma}\label{sec:attachment:lem:self-diffeo}$(1)$ The $4$-manifolds $S_1$ and $S_2$ in Figure~\ref{fig5} are diffeomorphic to $S^2\times D^2$.\smallskip \\
$(2)$ Let $\widehat{f}: \partial S_1\to \partial S_2$ be the diffeomorphism induced by the obvious cork twist of $S_1$ $($i.e. exchanging the dot and $0)$. Note that the cork twist does not change the boundary of $S_1$. Then $\widehat{f}$ extends to a diffeomorphism between $S_1$ and $S_2$. 

\begin{figure}[ht!]
\begin{center}
\includegraphics[width=3.4in]{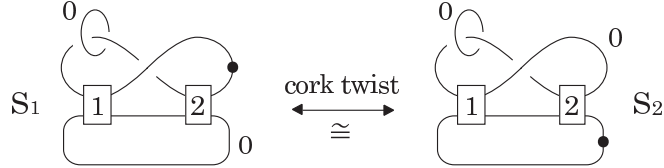}
\caption{two diagrams of $S^2\times D^2$}
\label{fig5}
\end{center}
\end{figure}
\end{lemma}
\begin{proof}$(1)$.\ The left side of the figure is checked by canceling the 1- and 2-handle pair. The right side is as follows. Slide the middle 2-handle over its meridian as in the second picture of Figure~\ref{fig6}. Note that the middle 2-handle now links with the dotted circle geometrically once. Canceling this 1- and 2-handle pair gives the last picture of the figure. 
\begin{figure}[ht!]
\begin{center}
\includegraphics[width=3.7in]{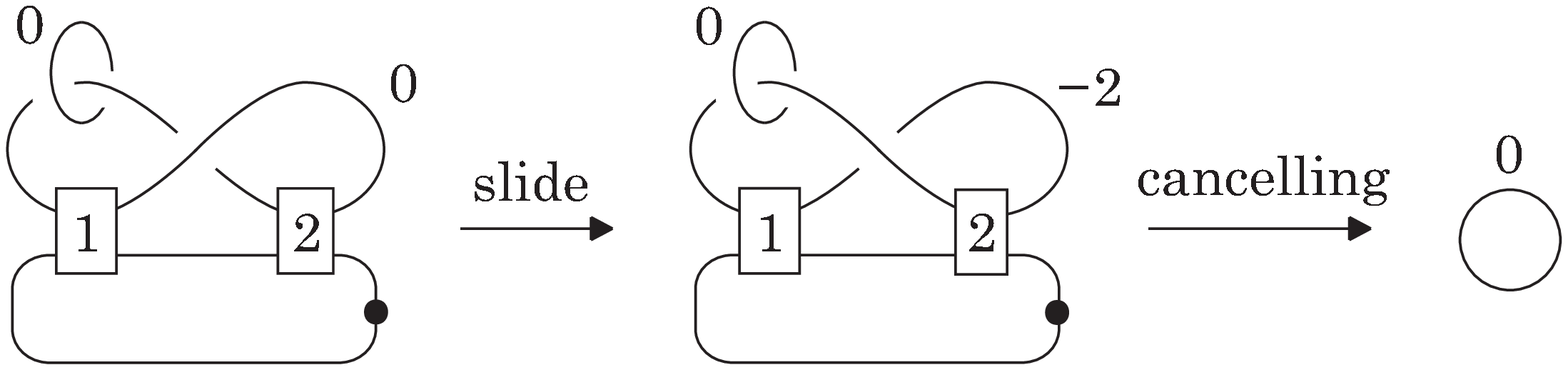}
\caption{}
\label{fig6}
\end{center}
\end{figure}

$(2)$.\ Since $f$ extends to a self-homeomorphism of $W$, $\widehat{f}$ extends to a homeomorphism between $S_1$ and $S_2$. Thus by $(1)$ and Gluck's theorem (Sections~5 and 15 of \cite{Gl2}, cf.\ \cite{DG}), $\widehat{f}$ extends to a diffeomorphism between $S_1$ and $S_2$.
\end{proof}

\begin{proposition}\label{sec:attachment:prop:embed}Let $K$ be a $2$-handle of a $2$-handlebody $X$. Let $Z$ be any compact connected oriented smooth $4$-manifold which  contains $X$ as a smooth submanifold. \smallskip \\
$(1)$ Let $X^+$ be the result of $X$ by a $W^+(p)$-modification to $K$. Then the following properties hold. 
\begin{itemize}
 \item [(i)] $X^+$ becomes diffeomorphic to $X$ after attaching a $2$- and a $3$-handle to $\partial X^+$ as in Figure~\ref{fig7}. Hence, $X^+$ can be embedded into $X$ and also $Z$.\smallskip
 \item [(ii)] The fundamental group, the integral homology groups, the integral homology groups of the boundary, and the intersection form of $Z-\textnormal{int}\,X^+$ are isomorphic to those of $Z-\textnormal{int}\,X$. Here we see $X^+$ as a submanifold of $Z$, through the embedding in \textnormal{(i)}. \smallskip 
\end{itemize}
$(2)$ Let $X^-$ be the result of $X^+$ by replacing the above $W^+(p)$-modification with the $W^-(p)$-modification as in the second row of Figure~\ref{fig8}. Then the following properties hold. 
\begin{itemize}
 \item [(i)] The cork twist of $Z$ along $(W, f)$ is diffeomorphic to $Z$ $($see Figure~\ref{fig8}$)$. Here this $W$ is the cork in $X^+(\subset Z)$ created from the $W^+(p)$-modification, and we view $X^+$ as a submanifold of $Z$ coming from the embedding in $(1)$\textnormal{(i)} above. Hence, $X^-$ can be embedded into $X$, and also into $Z$.\smallskip 
 \item [(ii)] $Z-\textnormal{int}\,X^-$ is diffeomorphic to $Z-\textnormal{int}\,X^+$. Here we see $X^+$ and $X^-$ as submanifolds of $Z$, via the embeddings in $(1)$.\textnormal{(i)} and $(2)$.\textnormal{(i)}, respectively.\smallskip
 \item [(iii)] $X$ can be embedded into $X^-$ so that the induced homomorphism $H_*(X;\mathbf{Z})\to H_*(X^-;\mathbf{Z})$ is an isomorphism.\smallskip 
\end{itemize}
$(3)$ There exist homeomorphisms between the pairs $(Z, X^+)$ and $(Z, X^-)$, and also between the pairs $(Z, Z-\textnormal{int}\,X^+)$ and $(Z, Z-\textnormal{int}\,X^-)$. 
\begin{figure}[ht!]
\begin{center}
\includegraphics[width=4.7in]{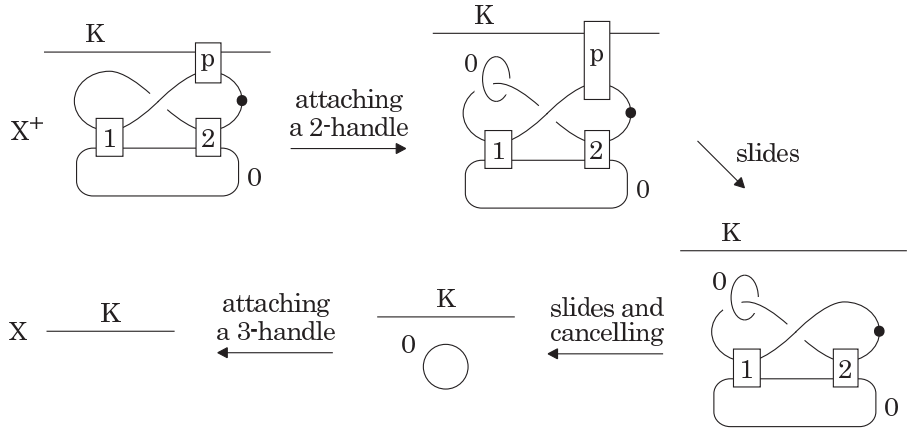}
\caption{Attaching a 2- and a 3-handle to $X^+$}
\label{fig7}
\end{center}
\end{figure}
\begin{figure}[ht!]
\begin{center}
\includegraphics[width=4.1in]{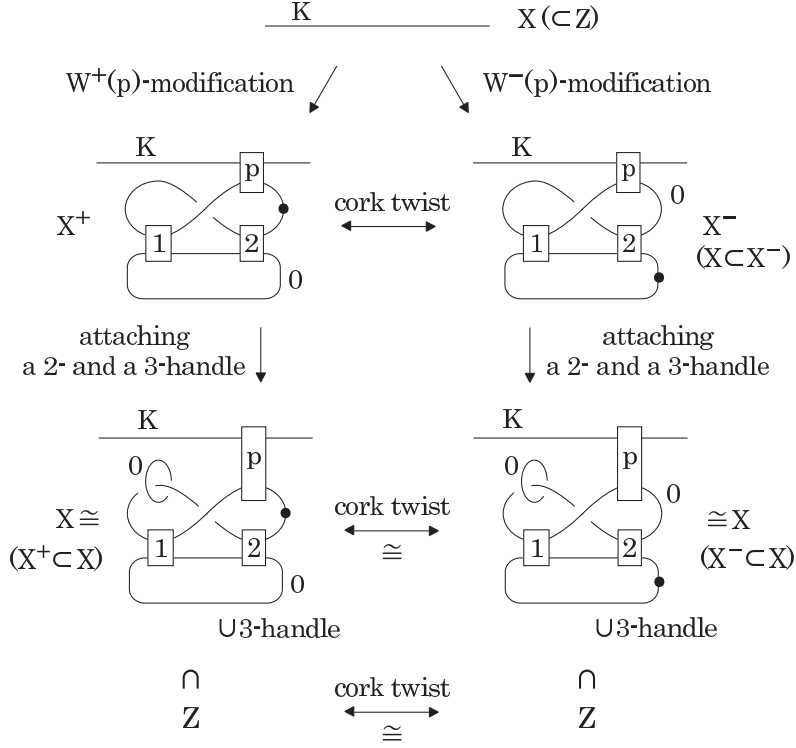}
\caption{relations}
\label{fig8}
\end{center}
\end{figure}
\end{proposition}
\begin{proof}$(1)$.(i).\ The first picture of Figure~\ref{fig7} is a local diagram of $X^+$. Following the procedure in the figure, we recover a diagram of $X$. Hence the claim follows. \smallskip

(1).(ii).\ Reverse the procedure in Figure~\ref{fig7} untill the second picture, keeping track of the 3-handle introduced in the fifth picture. Then we see, in the second picture of this figure, that the attaching sphere of the 3-handle intersects with the belt circle of the lower 0-framed 2-handle geometrically once. In the second picture, the 3-handle algebraically also cancels the upper 0-framed 2-handle. This is 
because the meridian of the upper 0-framed 2-handle in the second picture becomes homotopic to a curve linking the 0-framed unknot in the fourth picture geometrically once, although it also tangles around $K$.
Thus $Z-\textnormal{int}\,X^+$ is obtained from $Z-\textnormal{int}\,X$ by attaching a dual of this algebraically canceling 2- and 3-handle pairs (which is an algebraically canceling 1- and 2-handle pairs). 

This fact immediately gives the claim about the fundamental group, the homology groups and the intersection form. For the claim about the homology groups of the boundary, we first cap off the boundary of $Z$ to form a closed 4-manifold.  Now the claim follows from the fact above and the argument in the proof of Proposition~\ref{prop:relation of attachments}. \smallskip

$(2)$.\ Lemma~\ref{sec:attachment:lem:self-diffeo} gives (i). Note that the $Z=X$ case shows that $X^-$ (i.e.\ the cork twist of $X^+$) is embedded into $X$. Thus the complements of $X^+$ and $X^-$ in $Z$ are the same, and hence (ii) follows. Since $X^-$ is obtained from $X$ by attaching an algebraically canceling 1- and 2-handle pair, (iii) follows. \smallskip

$(3)$. By (2), the cork twist along $(W,f)$ changes $(Z, X^+)$ and $(Z, Z-\textnormal{int}X^+)$ into $(Z, X^-)$ and $(Z, Z-\textnormal{int}X^-)$, respectively. Since $f$ extends to a self homeomorphism of $W$, the claim follows. 
\end{proof}

Next we define Legendrian versions of $W^+$- and $W^-$-modifications for Legendrian handlebodies (recall Definition~\ref{sec:stein:def:handlebody}).\smallskip

Let $K$ be a 2-handle of a Legendrian handlebody. Take a small segment of the attaching circle of $K$ as in the first row of Figure~\ref{fig10}. Without loss of generality, we can assume that the orientation of the segment of $K$ is from the left to the right (Otherwise locally apply the Legendrian isotopy in Figure~\ref{fig9}. Note that this isotopy does not change the Thurston-Bennequin number and the rotation number).  
\begin{figure}[ht!]
\begin{center}
\includegraphics[width=2.3in]{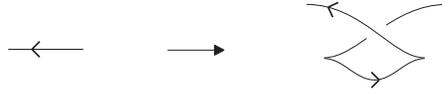}
\caption{Legendrian isotopy}
\label{fig9}
\end{center}
\end{figure}

\begin{definition}\label{sec:attachment:def:attachment} Let $p\geq 1$. We call the local operations shown in the left and the right side of Figure~\ref{fig10} a \textit{$W^+(p)$-modification} to $K$ and a \textit{$W^-(p)$-modification} to $K$, respectively. Here we orient the 2-handles as in the figure. Hence, each operation produces a new Legendrian handlebody from a given Legendrian handlebody. When we see Legendrian handlebodies as smooth handlebodies, these definitions and the orientations are consistent with those in Definition~\ref{sec:attachment:def:attachment} and Figure~\ref{fig3} (We can check this just by converting the 1-handle notation). Note that the auxiliary $2$-handle $\gamma$ to any $W^+(p)$- (resp.\ $W^-(p)$-) modification satisfies the following: its framing is $0$ (resp.\ $0$); $tb(\gamma)=2$ (resp.\ $tb(\gamma)=1$); $r(\gamma)=0$ (resp.\ $r(\gamma)=1$). 
\end{definition}
\begin{figure}[ht!]
\begin{center}
\includegraphics[width=3.9in]{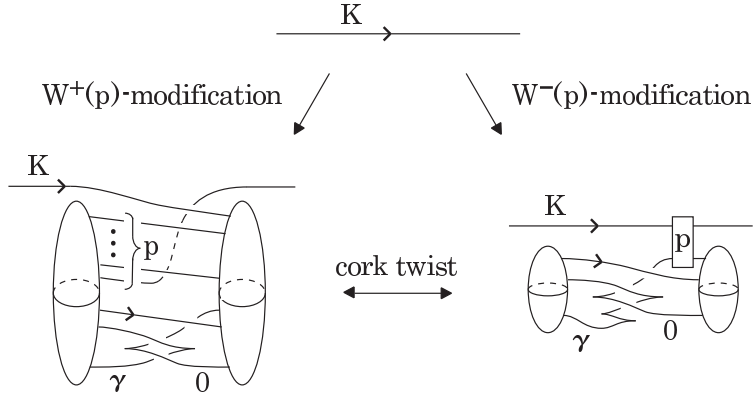}
\caption{$W^+(p)$- and $W^-(p)$-modification $(p\geq 1)$. Every framing is Seifert framing. The framing of $K$ is unchanged.}
\label{fig10}
\end{center}
\end{figure}
The above definition clearly shows the following.
\begin{proposition}\label{sec:attachment:prop:tb and r}Let $K$ be a $2$-handle of a Legendrian handlebody. \smallskip \\
$(1)$ Every $W^+(p)$-modification to $K$ has the following effect.
\begin{itemize}
 \item $tb(K)$ is increased by $p$, and $r(K)$ is unchanged. \smallskip
\end{itemize}
$(2)$ Every $W^-(p)$-modification to $K$ has the following effect.  
\begin{itemize}
 \item $tb(K)$ and $r(K)$ are unchanged. 
\end{itemize}
\end{proposition}
\begin{remark}For simplicity we used only $(W_1,f_1)$ for $W$-modifications. Many other corks, including $(W_n,f_n)$, also work similarly. For example,  the operation of ``creating a positron'' (together with its cork twist)  introduced by Akbulut-Matveyev~\cite{AM2}  has similar effects. An important effect of $W$-modifications is to increase the ``minimal genera'' of second homology classes (under some conditions). This is implied in the next section, through the proof of Theorem~\ref{sec:intro:th:main}. Essentially different operations (e.g.\ band sum with a knot with a sufficiently large Thurston-Bennequin number) also have this effect, though they do not share some other effects. 
\end{remark}

\section{Exotic Stein 4-manifolds and exotic embeddings}\label{sec:main_algorithm}
\subsection{Construction}\label{subsec:main_algorithm}
Here  we give an algorithm which provides Theorems~\ref{sec:intro:th:main} and \ref{sec:into:th:exotic embedding}. Later in subsection~\ref{subsec:simplestexample} we demonstrate this algorithm on a simple example.
\begin{definition}\label{sec:main:def:X}
Let $X$ be a compact oriented 4-dimentional $2$-handlebody with $b_2(X)\geq 1$. 
Throughout this section we fix this $X$. Let $k:=b_2(X)-1$. 
\end{definition}
Now we begin with  the construction. Recall the definitions of Legendrian and Stein handlebodies in Definition~\ref{sec:stein:def:handlebody}. Apply the following Step~\ref{step1} to $X$. 

\begin{step}\label{step1}
Slide and isotope the handles of $X$ so that $X$ is a Legendrian handlebody and that its 2-handles satisfy the following condition. \smallskip 
\begin{itemize}
 \item [$\bullet$]$2$-handles $K_j$ $(0\leq j\leq k)$ of $X$ do not algebraically go over any $1$-handle. So the second homology classes of $X$ given by the 2-handles $K_j$ $(0\leq j\leq k)$ span a basis of $H_2(X;\mathbf{Z})$. Here  $K_j$ $(0\leq j\leq l)$ denote all the $2$-handles of $X$ $(l\geq k)$. \smallskip 
 \end{itemize}
\end{step}

We use the following terminiogy. 
\begin{definition}\label{sec:main:def:good stein}
We call a Stein handlebody a \textit{good Stein handlebody} if it satisfies the condition described in Step~\ref{step1}. 
\end{definition}
\begin{remark}[The outline of the algorithm] Here we briefly summarize the algorithm. (However, beware that the actual construction is rather different.)\smallskip \\
$(1)$ The $b_2(X)=1$ case:  Apply $W^-(p_1)$-, $W^-(p_2)$-, \dots, $W^-(p_n)$-modifications to $K_0$ of $X$ and call the result $X^{(n)}_0$ (Figure~\ref{fig11}). Then replace the above $W^-(p_i)$-modification with the corresponding $W^+(p_i)$-modification, and denote the resulting manifold by $X^{(n)}_i$. If we choose $p_1\ll p_2\ll \dots\ll p_n$ as sufficiently large integers, then the minimal genera of the second homology classes given by $K_0$ (after sliding over the auxiliary 2-handle $p_i$ times) in $X^{(n)}_i$ $(0\leq i\leq n)$ become mutually different. We check this using the adjunction inequalities in the Stein manifolds $X^{(n)}_i$ $(1\leq i\leq n)$. We also apply $W^+(q_j)$-modifications to other 2-handles $K_j$ $(1\leq j\leq l)$ of $X$ at the beginning (Figure~\ref{fig12}.). In short the minimal genera detect the smooth structures of $X^{(n)}_i$ $(0\leq i\leq n)$. \smallskip\\
$(2)$ The $b_2(X)\geq 2$ case: This case is a generalization of the $b_2(X)=1$ case and it is more technical. In this case, we further adjust the rotation number of each $K_j$ $(1\leq j\leq k)$ by the above $W^+(q_j)$-modification to prevent these handles affecting adjunction inequality arguments. To detect smooth structures, we discuss the minimal genera of bases of $H_2(X^{(n)}_i;\mathbf{Z})$, using adjunction inequalities. 
\end{remark}


To proceed with the construction we need the following basic data for $X$.
\begin{definition}\label{sec:main:def:main data}Denote by $m_j, r_j, t_j$ $(0\leq j\leq l)$, the framing, the rotation number, and the Thurston-Bennequin number of $K_j$ of $X$, respectively. Let $g_j$ $(0\leq j\leq k)$ be the genus of a smoothly embedded surface in  the sub $1$-handlebody of $X$ spanned by $K_j$. Note that the attaching circle of every $K_j$ $(0\leq j\leq k)$ spans a surface because algebraically it does not  go over any of the $1$-handles (cf.~\cite{GS}). 
\end{definition}


Using this data, we here define integers for the construction. Roughly speaking, the following conditions require the each integer to be sufficiently large. 
\begin{definition}\label{sec:main:def:q_j}Put $q_0=0$. In the $l\geq 1$ case, define non-negative integers $q_j$ $(1\leq j\leq l)$ so that they satisify the following conditions. 
\begin{itemize}
 \item [(i)] $q_j+(t_j-1)-m_j\geq 0$, for each $1\leq j\leq l$\smallskip
 \item [(ii)] $q_j+(t_j-1)-m_j\geq \lvert r_j\rvert$, for each $1\leq j\leq k$ (in the $k\geq 1$ case).
\end{itemize}
\end{definition}
\begin{definition}\label{sec:the algorithm: def: p_i}Put $p_{-1}=p_0=0$. 
Define an increasing integer sequence $p_i$ $(i\geq 1)$ so that it satisfies the following conditions. 
\begin{itemize} 
\item [(i)] $p_i\,>\,p_{i-1}$, for each $i\geq 1$. \smallskip
 \item [(ii)]  $p_1+(t_0-1)-m_0\geq 0$. \smallskip
 \item [(iii)] $2p_1+(t_0-1)-m_0+\rvert r_0\lvert +m_j> 2(g_j+q_j)-2$, for each $0\leq j\leq k$.\smallskip 
 \item  [(iv)] $2p_i+(t_0-1)+\rvert r_0\lvert\, > 2(g_0+p_{i-1})-2$, for each $i\geq 1$. 
\end{itemize}
\end{definition}
\begin{remark}$(1)$ In the case where $t_0-1+\rvert r_0\lvert=2g_0-2$, the condition (iv) in Definition~\ref{sec:the algorithm: def: p_i} reduces to (i). 
\smallskip\\
$(2)$ In Definitions~\ref{sec:main:def:main data}, \ref{sec:main:def:q_j} and~\ref{sec:the algorithm: def: p_i}, we do not require neither the maximalities nor the minimalities of those numbers, therefore we can easily define those numbers.\smallskip\\
$(3)$ We don't need to calculate $g_j$ and $r_j$ for $k+1\leq j\leq l$, we do not use them. 
\end{remark}

We next adjust the Thurston-Bennequin numbers (and the rotation numbers) of 2-handles except $K_0$. Figures~\ref{fig11}--\ref{fig13} describe the local operations applied to 2-handles $K_j$ $(0\leq j\leq l)$ of $X$, through the following Steps \ref{step2}--\ref{step5} (without specifying Legendrian diagrams). 

\begin{definition}\label{sec:algorithm:def:$widetilde{X}$} Let $\widehat{X}$ be the Legendrian handlebody obtained from $X$ by applying the above Step~\ref{step1} and the Step~\ref{step2} below. (Skip Step~\ref{step2} when $l=0$.) 

\begin{step}\label{step2} Apply a $W^+(q_j)$-modification and add zig-zags to each $2$-handle $K_j$ $(1\leq j\leq l)$ of $X$ so that the following conditions are satisfied (recall Proposition~\ref{sec:attachment:prop:tb and r}, Lemma~\ref{sec:stein:lem:zigzag} and the conditions of $q_j$). Let $\delta_j$ $(1\leq j\leq l,\, q_j\neq 0)$ be the auxiliary $2$-handle to the above $W^+(q_j)$-modification. In the $l>k$ case, also add a zig-zag to each $\delta_j$ $(k+1\leq j\leq l)$ as follows (ignore (iii) when $k=l$). \smallskip
 \begin{itemize}
 \item [(i)] $tb(K_j)=m_j+1$ $(1\leq j\leq l)$. \smallskip
 \item [(ii)] $\lvert r(K_j)\rvert\leq 1$ $(1\leq j\leq k)$. \smallskip
 \item [(iii)] $tb(\delta_j)=1$ $(k+1\leq j\leq l,\, q_j\neq 0)$
\end{itemize}
\end{step}
\end{definition}

\begin{remark}
Note that the Thurston-Bennequin number of every $2$-handle of $\widehat{X}$ except $K_0$ and all of $\delta_j$ $(1\leq j\leq k,\, q_j\neq 0)$ is one more than its framing. 
\end{remark}

In the rest of this section, fix a positive integer $n$. 


\begin{definition}\label{sec:the algorithm:def:X}Apply Steps~\ref{step3}, \ref{step4} and~\ref{step5},  define $X^{(n)}_0$ then $X^{(n)}_i$ and $\gamma_i$ as follows. 
\begin{step}\label{step3}
Define $X^{(n)}_{0}$ as the Legendrian handlebody obtained from $\widehat{X}$ by applying $W^-(p_1)$-, $W^-(p_2)$-, \dots, $W^-(p_n)$-modifications to the $2$-handle $K_0$. 
\end{step}
\begin{step}\label{step4} In the $l\geq 1$ case, define $X^{(n)}_{-1}$ as the Legendrian handlebody obtained from $X^{(n)}_0$ by replacing every $W^+(q_j)$-modification $(1\leq j\leq l,\, q_j\neq 0)$ applied in Step~\ref{step2} with the corresponding $W^-(q_j)$-modification. In this case, we also skip the zig-zag operations in Step~\ref{step2}. In the $l=0$ case, put $X^{(n)}_{-1}:=X^{(n)}_{0}$.
\end{step}
\begin{step}\label{step5} Define $X^{(n)}_i$ $(1\leq i\leq n)$ as the Legendrian handlebody obtained from $X^{(n)}_0$ by replacing the $W^-(p_i)$-modification applied in Step~\ref{step3} with the corresponding $W^+(p_i)$-modification.  Let $\gamma_i$ $(1\leq i\leq n)$ denote the auxiliary $2$-handle of $X^{(n)}_i$ to the above $W^+(p_i)$-modification.  By adding zig-zags to $K_0$ and $\gamma_i$, we can assume that $K_0$ and $\gamma_i$ of the each Legendrian handlebody $X^{(n)}_i$ $(1\leq i\leq n)$ satisfy the following conditions (i)--(v) (recall Proposition~\ref{sec:attachment:prop:tb and r} and Lemma~\ref{sec:stein:lem:zigzag}.). (Namely, we add zig-zags so that the value $\lvert \langle c_1(X^{(n)}_i),\, [K_0-p_i\gamma_i] \rangle\rvert$ becomes as large as possible, see Lemma~\ref{lem:change of stein structure}). 

\begin{itemize}
 \item [(i)] $tb(K_0)=m_0+1$.\smallskip 
 \item [(ii)] $\lvert r(K_0)\rvert=p_i+(t_0-1)-m_0+\lvert r_0\rvert$. \smallskip 
 \item [(iii)] $tb(\gamma_i)=1$. \smallskip 
 \item [(iv)] $\lvert r(\gamma_i)\rvert=1$. \smallskip 
 \item [(v)] In the $r(K_0)\neq 0$ case, the sign of $r(\gamma_i)$ is opposite to the sign of $r(K_0)$.
\end{itemize}
\end{step}
\end{definition}

We now finished the construction and here discuss Stein structures on $X^{(n)}_i$.

\begin{remark}\label{rem:$X^{(n)}_i$}$(1)$ If $b_2(X)=1$ or $(q_1,q_2,\dots,q_k)= 0$, then $X^{(n)}_i$ $(1\leq i\leq n)$ is now a Stein handlebody. \smallskip\\
$(2)$ In the case where $b_2(X)\geq 2$ and $(q_1,q_2,\dots,q_k)\neq 0$, $X^{(n)}_i$ $(1\leq i\leq n)$ is not a Stein handlebody yet, because the Thurston-Bennequin number of each $\delta_j$ $(1\leq j\leq k,\, q_j\neq 0)$  is still two more than its framing. We can make each $tb(\delta_j)$ one more than its framing, by adding a zig-zag either upward or downward. Correspondingly, $r(\delta_j)$ becomes $-1$ or $1$. This process gives various Stein structures on each $X^{(n)}_i$ $(1\leq i\leq n)$. We later use this flexibility of Stein structures to simplify adjunction inequality arguments. \smallskip \\ 
$(3)$ By adding a zig-zag to each $\delta_j$, $X^{(n)}_0$ becomes a Stein handlebody when the original $X$ is a good Stein handlebody. \smallskip \\ 
$(4)$ $X^{(n)}_{-1}$ is a good Stein handlebody when the original $X$ is a good Stein handlebody.\smallskip \\
$(5)$ As a smooth handlebody, $X^{(n)}_{-1}$ is obtained from $X$ only by $W^-$-modifications. Proposition~\ref{sec:attachment:prop:embed} thus shows that $X$  can be embedded into $X^{(n)}_{-1}$ so that the induced homomorphism $H_*(X;\mathbf{Z})\to H_*(X^{(n)}_{-1};\mathbf{Z})$ is an isomorphism.
\end{remark}
\begin{figure}[ht!]
\begin{center}
\includegraphics[width=4.6in]{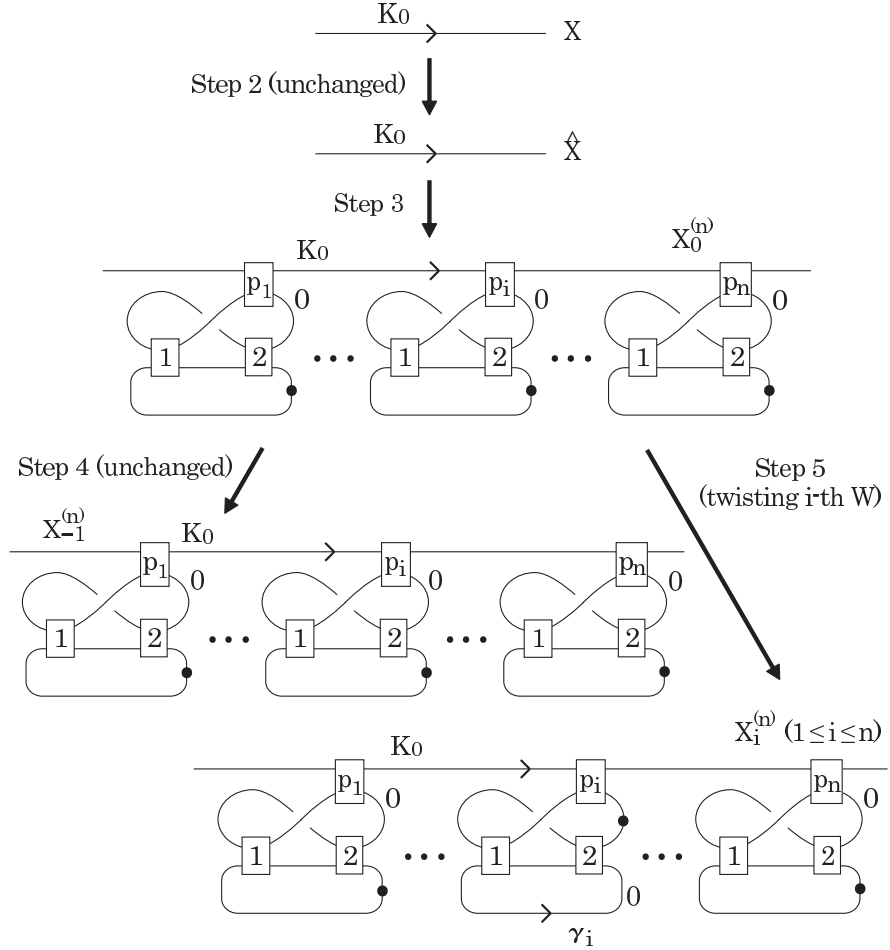}
\caption{Operations to $K_0$ (ignoring Legendrian diagrams)}
\label{fig11}
\end{center}
\end{figure}
\begin{figure}[ht!]
\begin{center}
\includegraphics[width=4.3in]{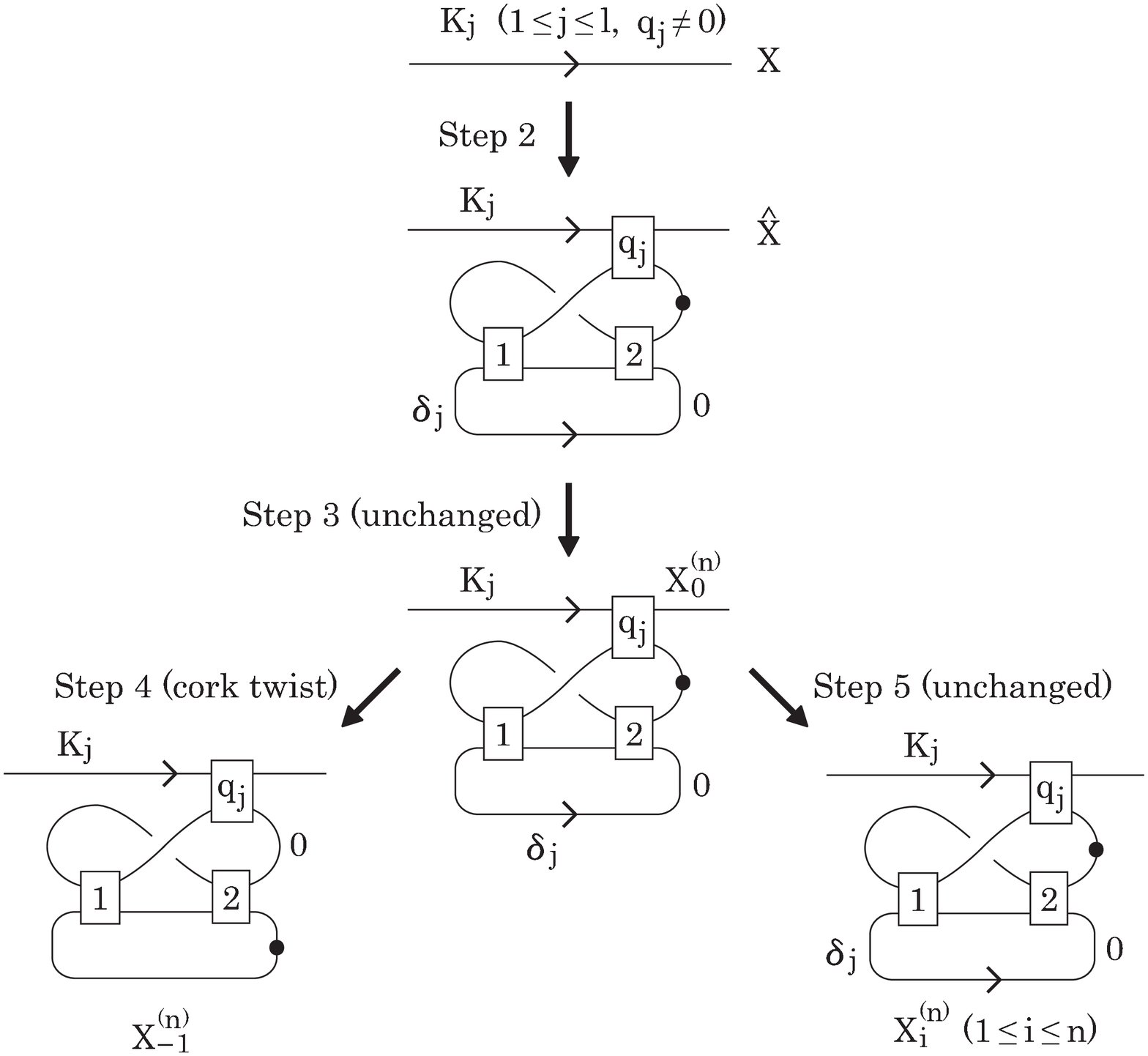}
\caption{Operations to $K_j$ $(1\leq j\leq l,\, q_j\neq 0)$ (ignoring Legendrian diagrams)}
\label{fig12}
\end{center}
\end{figure}
\begin{figure}[ht!]
\begin{center}
\includegraphics[width=3.5in]{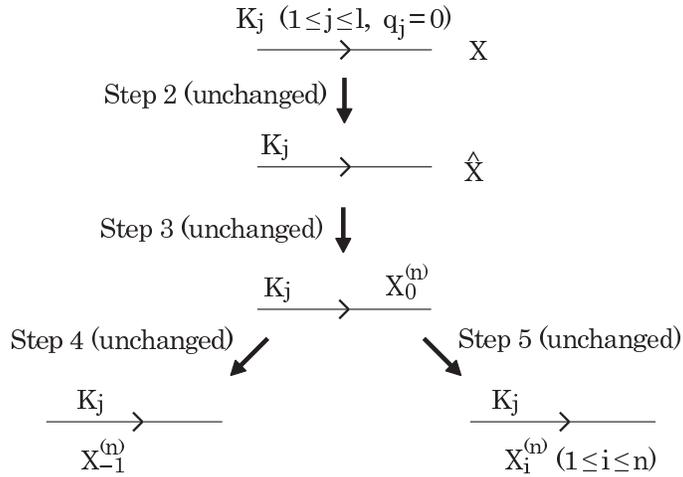}
\caption{Operations to $K_j$ $(1\leq j\leq l,\, q_j= 0)$ (ignoring Legendrian diagrams)}
\label{fig13}
\end{center}
\end{figure}
\subsection{Detecting smooth structures}
We next detect the smooth structures of $X^{(n)}_i$'s, examining the genera of bases of their second homology groups. 
\begin{definition}\label{sec:algorithm:def:v}Ignore $(2)$ when $b_2(X)=1$.\smallskip \\
$(1)$ Define $v^{(0)}_0$ as the element of $H_2(X^{(n)}_0;\mathbf{Z})$ given by the 2-handle $K_0$. Define $v^{(i)}_0$ $(1\leq i\leq n)$ as the element of $H_2(X^{(n)}_i;\mathbf{Z})$ given by the 2-handle $K_0-p_i\gamma_i$, which denotes the result of $K_0$ by sliding over $\gamma_i$ $p_i$-times so that it does not algebraically go over any 1-handle. \smallskip \\
$(2)$ For each $0\leq i\leq n$, define $v^{(i)}_j$ $(1\leq j\leq k)$ as the elements of $H_2(X^{(n)}_i;\mathbf{Z})$ given by the 2-handles $K_j-q_j\delta_j$. Here $K_j-q_j\delta_j$ denotes the result of $K_j$ by sliding over $\delta_j$ $q_j$-times so that it does not algebraically go over any 1-handle.  \smallskip\\ 
$(3)$ Let $v^{(-1)}_j$ $(0\leq j\leq k)$ as the element of $H_2(X^{(n)}_{-1};\mathbf{Z})$ given by the 2-handle $K_j$. 
\end{definition}


The Lemma below clearly follows from Proposition~\ref{sec:attachments:prop:genus of attachments}. 


\begin{lemma}\label{sec:main:lem:genus}
$(1)$ For each $0\leq i\leq n$, the elements $v^{(i)}_j$ $(0\leq j\leq k)$ span a basis of $H_2(X^{(n)}_i;\mathbf{Z})$ and satisfy the following conditions $($ignore \textnormal{(ii)} when $k=0$$)$.\smallskip
\begin{itemize}
 \item [(i)] $v^{(i)}_0$ is represented by a smoothly embedded genus $g_0+p_i$ surface, satisfying $v^{(i)}_0\cdot v^{(i)}_0=m_0$.\smallskip 
 \item [(ii)] Each $v^{(i)}_j$ $(1\leq j\leq k)$ is represented by a smoothly embedded genus $g_j+q_j$ surface, satisfying $v^{(i)}_j\cdot v^{(i)}_j=m_j$.
\end{itemize}
$(2)$ $v^{(-1)}_j$ $(0\leq j\leq k)$ span a basis of $H_2(X^{(n)}_{-1};\mathbf{Z})$, and each $v^{(-1)}_j$ $(0\leq j\leq k)$ is represented by a smoothly embedded genus $g_j$ surface satisfying $v^{(-1)}_j\cdot v^{(-1)}_j=m_j$.
\end{lemma}

We here use the flexibility of Stein structures on $X^{(n)}_i$ $(1\leq i\leq n)$. 
\begin{lemma}\label{lem:change of stein structure}For each integers $a_0, a_1, \dots, a_k$, there exists a Stein structure $J$ on the each smooth $4$-manifold $X^{(n)}_i$ $(1\leq i\leq n)$ such that
\begin{align*}
\lefteqn{\lvert \langle c_1(X^{(n)}_i,J),\, a_0v^{(i)}_0+a_1v^{(i)}_1+\dots+a_kv^{(i)}_k \rangle\rvert} \\
 &\geq\lvert a_0\rvert (2p_i+(t_0-1)-m_0+\lvert r_0\rvert)+ \lvert a_1\widehat{q}_1\rvert +\lvert a_2\widehat{q}_2\rvert +\dots+\lvert a_k\widehat{q}_k\rvert, 
\end{align*}
where $\widehat{q}_j=q_j-1$ $(\textnormal{if}\;\, q_j\neq 0)$ and $\widehat{q}_j=0$ $(\textnormal{if}\;\, q_j= 0)$. 
Furthermore, the equality holds  in the $k=0$ case $($ignoring the last $k$ terms$)$.
\end{lemma}
\begin{proof}Recall Steps~\ref{step2} and \ref{step5} and Remark~\ref{rem:$X^{(n)}_i$}.(2). By appropriately adding a zig-zag to each $\delta_j$ $(1\leq j\leq k)$ of $X^{(n)}_i$, Theorem~\ref{sec:stein:th:stein_rotation} easily gives the required claim.
\end{proof}
\begin{proposition}\label{sec:main theorem:prop:genus}For any $1\leq i\leq n$, there exists no basis $u_0,u_1,\dots,u_k$ of $H_2(X^{(n)}_i;\mathbf{Z})$ which satisfies the following conditions $($ignore \textnormal{(ii)} when $b_2(X)=1$$)$.\smallskip
\begin{itemize}
 \item [(i)] $u_0$ is represented by a smoothly embedded surface with its genus equal to or less than $g_0+p_{i-1}$ and satisfies $u^2_0=m_0$.\smallskip
 \item [(ii)] Each $u_j$ $(1\leq j\leq k)$ is represented by a smoothly embedded surface of genus $g_j+q_j$ and satisfies $u^2_j=m_j$.
\end{itemize}
\end{proposition}
\begin{proof}
Fix $i$ with $1\leq i\leq n$. 
Suppose that a basis $u_0,u_1,\dots,u_k$ of $H_2(X^{(n)}_i;\mathbf{Z})$ satisfies the above conditions (i) and (ii). We can assume that the genus of $u_0$ is $g_0+p_{i-1}$, by taking a connected sum with a null-homologous surface in $X^{(n)}_i$. For each $0\leq j\leq k$, put $u_j:=a^{(j)}_0v^{(i)}_0+a^{(j)}_1v^{(i)}_1+\dots+a^{(j)}_kv^{(i)}_k$. Lemma~\ref{lem:change of stein structure} and the adjunction inequality for $u_0$ give the inequality below. 
\begin{align*}
\lefteqn{2(g_0+p_{i-1})-2\geq\lvert a^{(0)}_0\rvert (2p_i+(t_0-1)-m_0+\lvert r_0\rvert)} \\
 &\qquad \qquad \qquad \qquad \quad+\lvert a^{(0)}_1\widehat{q}_1\rvert+\lvert a^{(0)}_2\widehat{q}_2\rvert+\dots+\lvert a^{(0)}_k\widehat{q}_k\rvert+m_0.
\end{align*}
This inequality and the condition (iv) of $p_i$'s in Definition~\ref{sec:the algorithm: def: p_i} easily show the following.
\begin{equation*}
0>(\lvert a^{(0)}_0\rvert-1) (2p_i+(t_0-1)-m_0+\lvert r_0\rvert).
\end{equation*}
The conditions (i) and (ii) of $p_i$'s in Definition~\ref{sec:the algorithm: def: p_i} thus give $a^{(0)}_0=0$. When $k=0$, this fact contradicts the assumption, hence the required claim follows. We thus assume $k\geq 1$. \smallskip

Lemma~\ref{lem:change of stein structure} and the adjunction inequality for $u_j$ $(1\leq j\leq k)$ give the below. 
\begin{align*}
\lefteqn{2(g_j+q_{j})-2\geq\lvert a^{(j)}_0\rvert (2p_i+(t_0-1)-m_0+\lvert r_0\rvert)} \\
&\qquad \qquad \qquad \qquad +\lvert a^{(j)}_1\widehat{q}_1\rvert+\lvert a^{(j)}_2\widehat{q}_2\rvert+\dots+\lvert a^{(j)}_k\widehat{q}_k\rvert+m_j.
\end{align*}
This inequality and the condition (i) and (iii) of $p_i$'s in Definition~\ref{sec:the algorithm: def: p_i} easily give the following. 
\begin{equation*}
0>(\lvert a^{(j)}_0\rvert-1) (2p_i+(t_0-1)-m_0+\lvert r_0\rvert).
\end{equation*}
The conditions (i) and (ii) of $p_i$'s in Definition~\ref{sec:the algorithm: def: p_i} thus gives $a^{(j)}_0=0$. We thus have $a^{(0)}_0=a^{(1)}_0=\dots=a^{(k)}_0=0$. This fact contradicts the assumption that $u_0,u_1,\dots, u_k$ is a basis. Hence the required claim follows. 
\end{proof}

To summarize, here we list up properties of $X^{(n)}_i$. Beware that $X_0$ in Theorem~\ref{sec:intro:th:main} corresponds to $X^{(n)}_{-1}$ in this section and that $X_i$ $(1\leq i\leq n)$ corresponds to $X^{(n)}_{i}$. 
 
\begin{theorem}\label{sec:algorithm1:main theorem}Let $X$ be any $2$-handlebody with $b_2(X)\geq 1$. Fix $n\geq 1$. Let $X^{(n)}_i$ $(-1\leq i\leq n)$ denote the corresponding Legendrian handlebodies in Definition~\ref{sec:the algorithm:def:X}. Then the following properties hold. \smallskip \\
$(1)$ The fundamental group, the integral homology groups, the integral homology groups of the boundary, and the intersection form of each $X^{(n)}_i$ $(-1\leq i\leq n)$ are isomorphic to those of $X$.\smallskip \\
$(2)$ $X^{(n)}_i$ $(0\leq i\leq n)$ are mutually homeomorphic but not diffeomorphic with respect to the given orientations. When either the following \textnormal{(i)} or \textnormal{(ii)} holds, they are mutually non-diffeomorphic with any orientations. The same properties also hold for $X^{(n)}_i$ $(-1\leq i\leq n,\; i\neq 0)$.\smallskip 
\begin{itemize}
 \item [(i)] $b_2(X)=1$.\smallskip 
 \item [(ii)] The intersection form of $X$ is represented by the zero matrix. \smallskip
\end{itemize}
$(3)$ Each $X^{(n)}_i$ $(1\leq i\leq n)$ admits a Stein structure. $X^{(n)}_{-1}$ and $X^{(n)}_{0}$ admit Stein structures when $X$ is a good Stein handlebody. \smallskip \\
$(4)$ $X$ can be embedded into $X^{(n)}_{-1}$ so that the induced homomorphism is an isomorphism between the integral homology groups of $X$ and $X^{(n)}_{-1}$. Therefore $X^{(n)}_{-1}$ does not admit any Stein structure when $X$ cannot be embedded into any simply connected minimal symplectic $4$-manifold with $b^+_2>1$ $($or any minimal complex surface of general type with $b^+_2>1$$)$ so that the induced homomorphism between the second homology groups is injective.\smallskip\\
$(5)$ Each $X^{(n)}_i$ $(-1\leq i\leq n)$ can be embedded into $X$. \smallskip\\
$(6)$ There exist disjoint copies $C_i$ $(1\leq i\leq n)$ of $W_1$ in $X^{(n)}_0$ such that $X^{(n)}_i$ is the cork twist of $X^{(n)}_0$ along $(C_i, f_1)$. 
\end{theorem}
\begin{proof}Proposition~\ref{prop:relation of attachments} gives $(1)$. For $(3)$, see Remark~\ref{rem:$X^{(n)}_i$}. Proposition~\ref{sec:attachment:prop:embed} gives $(5)$. Construction shows $(6)$.\smallskip


We next check $(2)$. $X^{(n)}_i$ $(-1\leq i\leq n)$ are mutually homeomorphic because they are related to each other by combinations of cork twists. Since $p_i$ $(i\geq 0)$ is a strictly increasing sequence, Lemma~\ref{sec:main:lem:genus} and Proposition~\ref{sec:main theorem:prop:genus} show the first claim. The second claim in the (ii) case also follows from Lemma~\ref{sec:main:lem:genus} and Proposition~\ref{sec:main theorem:prop:genus}. In the case $b_2(X)=1$ and $m_0\neq 0$, there are no orientation-reversing homeomorphisms between them. Hence they cannot be orientation-reversing diffeomorphic.\smallskip

Lastly we show $(4)$. Remark~\ref{rem:$X^{(n)}_i$}.$(5)$ gives the first claim of $(4)$. Suppose that $X^{(n)}_{-1}$ admits a Stein structure. Then $X^{(n)}_{-1}$ admits a Stein handlebody presentation. For every 2-handle of this Stein handlebody, attach a 2-handle along its $-2$-framed meridian so that the result is also a Stein handlebody. This new Stein handlebody can be embedded into a simply connected minimal symplectic $4$-manifold and a minimal complex surface of general type (see Theorem~\ref{sec:stein:th:closing of Stein}). Note that, in each of this closed 4-manifold, the image of every non-zero second homology class of $X^{(n)}_{-1}$ algebraically intersects with a sphere with its self-intersection number $-2$. This fact implies the injectivity of the induced homomorphism between the second homology groups of $X^{(n)}_{-1}$ and the closed manifold. The second claim of $(4)$ thus easily follows.
\end{proof}

For a given embedding of a 4-manifold, applying the algorithm to its complement, we get arbitrary many exotic embeddings. In the following, beware that $Y_0$ in Theorem~\ref{sec:into:th:exotic embedding} corresponds to $Y^{(n)}_{-1}$ and that $Y_i$ $(1\leq i\leq n)$ corresponds to $Y^{(n)}_{i}$. 

\begin{theorem}\label{sec:algorithm1:exotic knottings}Let $Z$ and $Y$ be compact connected oriented smooth $4$-manifolds $($possibly with boundary$)$. Suppose that $Y$ is embedded into $Z$ and that its complement $X:=Z-\textnormal{int}\,Y$ is a $2$-handlebody with $b_2(X)\geq 1$. Fix $n\geq 1$. 
Then, there exist mutually diffeomorphic compact connected oriented smooth $4$-manifolds $Y^{(n)}_i$ $(-1\leq i\leq n)$  embedded into $Z$ with the following properties. \smallskip\\
$(1)$ The pairs $(Z, Y^{(n)}_i)$ $(0\leq i\leq n)$ are mutually homeomorphic but non-diffeomorphic with respect to the given orientations. When either the following \textnormal{(i)} or \textnormal{(ii)} holds, they are mutually non-diffeomorphic for any orientations. The same properties also hold for the pairs $(Z, Y^{(n)}_i)$ $(-1\leq i\leq n,\,i\neq 0)$.\smallskip 
\begin{itemize}
 \item [(i)] $b_2(X)=1$.\smallskip 
 \item [(ii)] The intersection form of $X$ is represented by the zero matrix. \smallskip
\end{itemize}
$(2)$ The fundamental group, the integral homology groups, the integral homology groups of the boundary, and the intersection form of each $Y^{(n)}_i$ $(-1\leq i\leq n)$ are isomorphic to those of $Y$. \smallskip \\
$(3)$ Each complement $X^{(n)}_i:=Z-\textnormal{int}\,Y^{(n)}_i$ $(-1\leq i\leq n)$ is the one in Theorem~\ref{sec:algorithm1:main theorem}, corresponding to the above $X$. 
\end{theorem}
\begin{proof}Let $X^{(n)}_i$ $(-1\leq i\leq n)$ denote the manifold in Theorem~\ref{sec:algorithm1:main theorem}, corresponding to the above $X$. Replace every $W^-$-modification of $X^{(n)}_0$ applied in Step~3 with the corresponding $W^+$-modification. As a smooth handlebody, the result of $X^{(n)}_0$ is thus obtained from $X$ only by $W^+$-modifications. Applying Proposition~\ref{sec:attachment:prop:embed}, embed this manifold into $X$. Proposition~\ref{sec:attachment:prop:embed} thus gives an embedding of each $X^{(n)}_i$ $(-1\leq i\leq n)$ into $X$ (and hence $Z$) by twisting $W$'s of the above manifold. Put $Y^{(n)}_i:=Z-\textnormal{int}\,X^{(n)}_i$ $(-1\leq i\leq n)$. Now the required claims easily follow from Proposition~\ref{sec:attachment:prop:embed} and Theorem~\ref{sec:algorithm1:main theorem}. 
\end{proof}
We now have Theorems~\ref{sec:intro:th:main} and \ref{sec:into:th:exotic embedding} and Corollary~\ref{sec:main:cor:cork:compact case}. 
\begin{proof}[Proofs of Theorems~\ref{sec:intro:th:main} and \ref{sec:into:th:exotic embedding} and Corollary~\ref{sec:main:cor:cork:compact case}] These clearly follow from Theorems~\ref{sec:algorithm1:main theorem} and \ref{sec:algorithm1:exotic knottings}. 
\end{proof}

\begin{remark}$(1)$ Suppose that $\partial X$ is diffeomorphic to $\#_m S^1\times S^2$ $(m\geq 0)$. Since $\#_m S^1\times S^2$ has a unique Stein filling, each $\partial X^{(n)}_i$ is not diffeomorphic to $\partial X$, so in particular, each $X^{(n)}_i$ is not homeomorphic to the original $X$, in this case. 
\smallskip \\
$(2)$ By more restricting the conditions of $p_i$ $(i\geq 1)$ in Definition~\ref{sec:the algorithm: def: p_i}, we can easily show the following: $X^{(n)}_0$ produces $2^{n}-1$ mutually homeomorphic but non-diffeomorphic compact Stein $4$-manifolds by natural combinations of cork twists. 
\end{remark}
\begin{remark}[Variants of the construction] There are many variants of the construction, here we  remark just a few of them.\smallskip \\
$(1)$ We can cut the condition (ii) of $q_j$ in Definition~\ref{sec:main:def:q_j}, by choosing each $p_i$ $(i\geq 1)$ larger. In this case, we use flexibility of zig-zag operations of $K_0$ and $\gamma_i$, instead of $\delta_j$. Namely, we equip each $X^{(n)}_i$ $(1\leq i\leq n)$ with two Stein structures so that $\langle c_1(X^{(n)}_i), v^{(i)}_0 \rangle$ takes two different values, namely, a large positive number and a large negative number. This makes us possible to apply similar adjunction inequality arguments. \smallskip \\
$(2)$ Though we used only $W_1$ for the construction, many other corks (e.g.\ $W_n$ of \cite{AY1}) also work. Taking band sums with knots with sufficiently large Thurston-Bennequin numbers are also helpful. We can use band sum operations in Step~\ref{step2}, instead of $W^+$-modifications, though Theorem~\ref{sec:algorithm1:main theorem}.(4) and Theorem~\ref{sec:algorithm1:exotic knottings} are not guaranteed in this case. 
\end{remark}
\begin{remark}Once we apply Step~1 to any given $X$ and calculate the data of $X$ in Definition~\ref{sec:main:def:main data}, we immediately get a (usually large) smooth handle diagram of each $X^{(n)}_i$, as shown in Figures~\ref{fig11}--\ref{fig13}. See Subsection~\ref{subsec:simplestexample}, for the simplest case. Though we can also immediately get a Legendrian (Stein) handlebody diagram of each $X^{(n)}_i$, it usually a very large diagram. 
\end{remark}
\section{Strengthening the construction}\label{sec:variant}

In Section~\ref{sec:main_algorithm}, we did not completely exclude the possibility that some of $X^{(n)}_i$'s are orientation-reversing diffeomorphic, because the argument was simplified and that the conditions of $p_i$'s in Definition~\ref{sec:the algorithm: def: p_i} were relaxed. 
In this section, we exclude this possibility by restricting the conditions of $p_i$'s. We use the same symbols as in Section~\ref{sec:main_algorithm}.
\smallskip
\begin{definition}\label{sec:variant:def}Let $X$ be any $2$-handlebody with $b_2\geq 1$. Fix $n\geq 1$. \smallskip\\
$(1)$ In the $b_2(X)=1$ case, put $\widehat{X}^{(n)}_i:=X^{(n)}_i$ $(-1\leq i\leq n)$, where $X^{(n)}_i$ are the manifolds as in Theorem~\ref{sec:algorithm1:main theorem}. \smallskip\\
$(2)$ In the $b_2(X)\geq 2$ case, assume that $p_i$'s $(i\geq -1)$ in Definition~\ref{sec:the algorithm: def: p_i} further satisfy the following conditions (v) and (vi). Then put $\widehat{X}^{(n)}_i:=X^{(n)}_i$ $(-1\leq i\leq n)$. \smallskip
\begin{itemize}
\item  [(v)] $2p_1+(t_0-1)-m_0+\rvert r_0\lvert -m_j> 2(g_j+q_j)-2$, for each $0\leq j\leq k$.\smallskip
\item  [(vi)] $2p_i+(t_0-1)-2m_0+\rvert r_0\lvert> 2(g_0+p_{i-1})-2$, for each $i\geq 1$. \smallskip
\end{itemize}
\end{definition}\smallskip

Let $v^{(i)}_j$ $(1\leq j\leq k)$ denote the basis of $H_2(\widehat{X}^{(n)}_{i};\mathbf{Z})$ in Definition~\ref{sec:algorithm:def:v}. Since we defined $\widehat{X}^{(n)}_i$ as a special case of $X^{(n)}_i$, the same properties as in Lemma~\ref{sec:main:lem:genus} hold. For Proposition~\ref{sec:main theorem:prop:genus}, we can easily get the following stronger claim. 
\begin{proposition}\label{sec:variant:prop:non-existence}Fix $n\geq 1$. For any $1\leq i\leq n$, there exists no basis $u_0,u_1,\dots,u_k$ of $H_2(\widehat{X}^{(n)}_i;\mathbf{Z})$ which satisfies the following conditions $($ignore \textnormal{(ii)} when $b_2(X)=1$$)$.\smallskip
\begin{itemize}
 \item [(i)] $u_0$ is represented by a smoothly embedded surface with its genus equal to or less than $g_0+p_{i-1}$ and satisfies $\lvert u^2_0\rvert =\lvert m_0\rvert $.\smallskip
 \item [(ii)] Each $u_j$ $(1\leq j\leq k)$ is represented by a smoothly embedded  surface of genus $g_j+q_j$ and satisfies $\lvert u^2_j\rvert =\lvert m_j\rvert $.
\end{itemize}
\end{proposition}
Using this proposition, we easily get the following strengthened theorems. 
\begin{theorem}\label{sec:variant:main theorem}Let $X$ be any $2$-handlebody with $b_2(X)\geq 1$. Fix $n\geq 1$. Let $\widehat{X}^{(n)}_i$ $(-1\leq i\leq n)$ denote the corresponding Legendrian handlebody in Definition~\ref{sec:variant:def}. Then the following properties hold. \smallskip \\
$(1)$ $\widehat{X}^{(n)}_i$ $(0\leq i\leq n)$ are mutually homeomorphic, but mutually non-diffeomorphic for any orientations. The same property also holds for $\widehat{X}^{(n)}_i$ $(-1\leq i\leq n,\, i\neq 0)$. \smallskip \\
$(2)$ $\widehat{X}^{(n)}_i$ $(-1\leq i\leq n)$ has the same properties as those of $X^{(n)}_i$ $(-1\leq i\leq n)$ in Theorem~\ref{sec:algorithm1:main theorem}.
\end{theorem}
\begin{theorem}\label{sec:variant:exotic knottings}Let $Z$ and $Y$ be compact connected oriented smooth $4$-manifolds $($possibly with boundary$)$. Suppose that $Y$ is embedded into $Z$ and that its complement $X:=Z-\textnormal{int}\,Y$ is a $2$-handlebody with $b_2(X)\geq 1$. Fix $n\geq 1$. 
Then, there exist mutually diffeomorphic compact connected oriented smooth $4$-manifolds $\widehat{Y}^{(n)}_i$ $(-1\leq i\leq n)$  embedded into $Z$ with the following properties. \smallskip\\
$(1)$ The pairs $(Z, \widehat{Y}^{(n)}_i)$ $(0\leq i\leq n)$ are mutually homeomorphic, but mutually non-diffeomorphic with any orientations. The same property also holds for the pairs $(Z, \widehat{Y}^{(n)}_i)$ $(-1\leq i\leq n,\,i\neq 0)$.\smallskip \\
$(2)$ The fundamental group, the integral homology groups, the integral homology groups of the boundary, and the intersection form of every $\widehat{Y}^{(n)}_i$ $(-1\leq i\leq n)$ are isomorphic to those of $Y$. \smallskip \\
$(3)$ Each complement $\widehat{X}^{(n)}_i:=Z-\textnormal{int}\,\widehat{Y}^{(n)}_i$ $(-1\leq i\leq n)$ is as in the Theorem~\ref{sec:variant:main theorem}, corresponding to the above $X$. 
\end{theorem}
\section{The contact structures on the boundary}\label{sec:contact}
In this section, we discuss the induced contact structures on the boundary $\partial X^{(n)}_i$ in the $b_2(X)=1$ case. We use the same symbols as in Section~\ref{sec:main_algorithm}. 
\begin{definition}Let $X$ be any $2$-handlebody with $b_2(X)=1$. Assume that the intersection form of $X$ is non-zero $($i.e.\ $m_0\neq 0$$)$. Fix $n\geq 1$. Let $X^{(n)}_i$ and $v^{(i)}_0$$(1\leq i\leq n)$ denote the corresponding Stein handlebody in Theorem~\ref{sec:algorithm1:main theorem} and the generator of $H_2(X^{(n)}_i;\mathbf{Z})$ in Definition~\ref{sec:algorithm:def:v}, respectively.  Let $\xi^{(n)}_i$ $(1\leq i\leq n)$ be the contact structure on $\partial X^{(n)}_i$ induced by the Stein structure on $X^{(n)}_i$.  
\end{definition}
\begin{lemma}$d_3(\xi^{(n)}_i)$ $(1\leq i\leq n)$ are mutually different. 
\end{lemma}
\begin{proof}Lemma~\ref{lem:change of stein structure} shows $\lvert \langle c_1(X^{(n)}_i), v^{(i)}_0 \rangle\rvert=2p_i+t_0-1-m_0+\lvert r_0\rvert$. The value $\lvert \langle c_1(X^{(n)}_i), v^{(i)}_0 \rangle\rvert$ hence strictly increases when $i$ increases. Lemma~\ref{sec:stein:lem:contact} thus show the claim. 
\end{proof}
The following proposition gives Corollary~\ref{sec:intro:cor:contact}. 
\begin{proposition}Fix $n\geq 1$. Each smooth $4$-manifold $X^{(n)}_i$ $(1\leq i\leq n)$ admits no Stein structure compatible with $\xi^{(n)}_j$ for any $j$ with $i<j\leq n$. 
\end{proposition}
\begin{proof}Suppose that some $X^{(n)}_{i}$ admits a Stein structure $J$ compatible with $\xi^{(n)}_{j}$ for some $j$ with $i<j\leq n$. Then the corresponding first Chern class $c_1(X^{(n)}_{i}, J)$ satisfies $\lvert \langle c_1(X^{(n)}_{i}, J), v^{(i)}_0 \rangle\rvert=\lvert \langle c_1(X^{(n)}_{j}), v^{(j)}_0 \rangle\rvert$, because $d_3(\xi^{(n)}_{j})$ is determined by this value (see Lemma~\ref{sec:stein:lem:contact}). The adjunction inequality for the Stein $4$-manifold $(X^{(n)}_{i}, J)$ thus easily shows that $v^{(i)}_0$ cannot be represented by any smoothly embedded surface with genus less than or equal to $g_0+p_{j-1}$ (see the proof of Proposition~\ref{sec:main theorem:prop:genus}). Since $p_{j-1}\geq p_{i}$,  this fact contradicts the fact that $v^{(i)}_0$ is represented by a surface of genus $g_0+p_{i}$ (see Lemma~\ref{sec:main:lem:genus}). 
\end{proof}

\section{Infinitely many disjoint corks in noncompact 4-manifolds}\label{sec:noncompact}
In~\cite{AY4}, we constructed infinitely many disjoint embeddings of the cork $W_1$ into a simply connected noncompact smooth $4$-manifold so that this noncompact 4-manifold produces infinitely many different exotic smooth structures by twisting the each copy of $W_1$. The second betti number of the noncompact 4-manifold is infinite. In this section, we construct such noncompact 4-manifolds for any finite second betti number larger than zero. Namely, we prove the following. 

\begin{theorem}\label{sec:noncompact:mainthm}Let $X$ be any compact $2$-handlebody with $b_2(X)\geq 1$. Then, there exist infinitely many noncompact $4$-manifolds $X_i$ $(i\geq 0)$ and infinitely many disjointly embedded copies $C_i$ $(i\geq 1)$ of $W_1$ into $X_0$ with the following properties.\smallskip\\
$(1)$ Each $X_i$ $(i\geq 1)$ is the cork twist of $X_0$ along $(C_i,f_1)$.\smallskip\\
$(2)$ $X_i$ $(i\geq 0)$ are mutually homeomorphic but not diffeomorphic. \smallskip\\
$(3)$ The fundamental group, the integral homology groups, and the intersection form of every $X_i$ $(i\geq 0)$ are isomorphic to those of $X$.\smallskip\\
$(4)$ Each $X_i$ $(i\geq 0)$ can be embedded into $X$. 
\end{theorem}

In this section, we use the same symbols as in Section~\ref{sec:variant} (and \ref{sec:main_algorithm}). Let $X$ be any compact $2$-handlebody with $b_2(X)\geq 1$. Let $\widehat{X}$ and $\widehat{X}^{(n)}_i$ $(0\leq i\leq n)$ denote the compact 2-handlebodies in Definitions~\ref{sec:algorithm:def:$widetilde{X}$} and \ref{sec:variant:def}, respectively, corresponding to this $X$. Recall that, for each $n\geq 1$, the smooth $4$-manifold $\widehat{X}^{(n)}_0$ is obtained from $\widehat{X}$ by attaching $n$ pairs of $1$- and $2$-handle to the boundary (i.e.\ by $W^{-}(p_{1})$-, $W^-(p_{2})$-, \dots, $W^-(p_{n})$-modifications to $K_0$). We thus have an infinite sequence $\widehat{X}^{(1)}_0\subset \widehat{X}^{(2)}_0\subset \dots \subset \widehat{X}^{(n)}_0\subset \widehat{X}^{(n+1)}_0\subset \cdots$. We are now ready to define noncompact $4$-manifolds. See also Figure~\ref{fig14}. 
\begin{figure}[ht!]
\begin{center}
\includegraphics[width=4.6in]{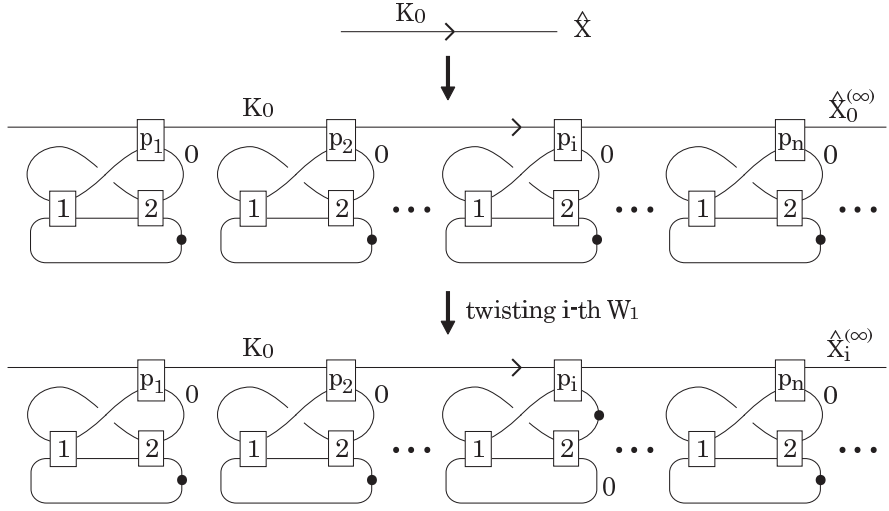}
\caption{$\widehat{X}^{(\infty)}_i$}
\label{fig14}
\end{center}
\end{figure}

\begin{definition}\label{sec:noncompact:def:X:infty}Let $\widehat{X}^{(\infty)}_0$ be the noncompact smooth $4$-manifold obtained by the inductive limit of the above sequence. Let $\widehat{X}^{(\infty)}_i$ $(i\geq 1)$ be the noncompact smooth $4$-manifold obtained from $\widehat{X}^{(\infty)}_0$ by the cork twist along $(W,f)$ where this $W$ is the one given by the above $W^-(p_{i})$-modification to $K_0$. 
\end{definition}


We can now easily prove the above theorem. 
\begin{proof}[Proof of Theorem~\ref{sec:noncompact:mainthm}]
Put $X_i:=\widehat{X}^{(\infty)}_i$. $(1)$ is obvious from the definition of $\widehat{X}^{(\infty)}_i$.\smallskip

$(2)$. Since each $\widehat{X}^{(\infty)}_i$ is obtained from $\widehat{X}^{(i+1)}_i$ by applying infinitely many $W^-$-modifications, Proposition~\ref{sec:attachment:prop:embed} implies that each $\widehat{X}^{(\infty)}_i$ $(i\geq 0)$ can be embedded into $\widehat{X}^{(i+1)}_i$ so that the induced homomorphism between the second homology groups is an isomorphism. (Thus this fact shows $(4)$.) It is thus easy to check that the claims similar to Lemma~\ref{sec:main:lem:genus} and Proposition~\ref{sec:variant:prop:non-existence} hold for $\widehat{X}^{(\infty)}_i$ $(i\geq 0)$. This shows that $\widehat{X}^{(\infty)}_i$ $(i\geq 0)$ are mutually non-diffeomorphic for any orientations. On the other hand, $\widehat{X}^{(\infty)}_i$ $(i\geq 0)$ are mutually homeomorphic because they are related by cork twists.\smallskip

$(3)$. The well-known properties of the inductive limit operation and Theorem~\ref{sec:variant:main theorem}.(1) show the $i=0$ case. Since $\widehat{X}^{(\infty)}_i$ $(i\geq 1)$ is a cork twist of $\widehat{X}^{(\infty)}_0$, the $i\geq 1$ case also follows.
\end{proof}
\begin{remark}The corresponding result holds when we use $X^{(n)}_i$ instead of $\widehat{X}^{(n)}_i$, though the orientation problem as in Theorem~\ref{sec:algorithm1:main theorem} remains, in this case.
\end{remark}
\section{Examples}\label{sec:example}
\subsection{The simplest example}\label{subsec:simplestexample}In this subsection, we apply the algorithm in Section~\ref{sec:main_algorithm} for the simplest example. Actually, our algorithm is a generalization of this example. We also demonstrate how to show the (non-) existence of Stein structures on $X^{(n)}_{-1}$ for the example. We use the same symbols as in Section~\ref{sec:main_algorithm}. \smallskip 
\begin{figure}[ht!]
\begin{center}
\includegraphics[width=0.5in]{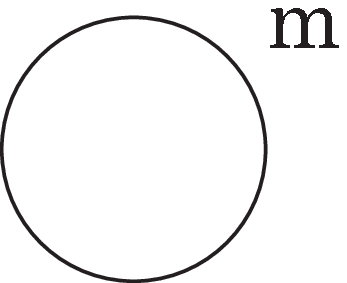}
\caption{$U(m)$}
\label{fig15}
\end{center}
\end{figure}
\begin{figure}[ht!]
\begin{center}
\includegraphics[width=0.9in]{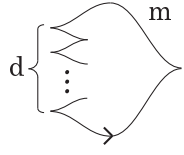}
\caption{A Legendrian handlebody presentation of $U(m)$}
\label{fig16}
\end{center}
\end{figure}

Let $U(m)$ be the 2-handlebody in Figure~\ref{fig15}. We here apply Step~\ref{step1} in Section~\ref{sec:main_algorithm} to this $U(m)$. There are infinitely many Legendrian handlebody presentation of $U(m)$, and we can use anyone. We here adopt the one in Figure~\ref{fig16}, where $d:=-m-1$ $(m\leq -2)$ and $d:=1$ $(m\geq -1)$. In this case, the symbols in Section~\ref{sec:main_algorithm} correspond to following: $X=U(m)$, and $K_0$ is the $m$-framed unknot in the figure, $k=0$, $m_0=m$, $t_0=-d$, $r_0=d-1$ and $g_0=0$. Put $p_{-1}=p_0=0$, $p_1=1$ $(m\leq -2)$, $p_1=m+2$ $(m\geq -1)$ and $p_i=p_1+i-1$ $(i\geq 2)$. This sequence $p_i$ $(i\geq -1)$ satisfies the conditions in Definition~\ref{sec:the algorithm: def: p_i}. \smallskip 

We next apply the Steps~\ref{step3}--\ref{step5} in Section~\ref{sec:main_algorithm} (Step~\ref{step2} is skipped in this case). Denote by $U(m)^{(n)}_i$ the Legendrian handlebody corresponding to $X^{(n)}_i$ for $U(m)$. A smooth handlebody diagram of $U(m)^{(n)}_i$ is given in Figures~\ref{fig17} and \ref{fig18}. Namely, $U(m)^{(n)}_i$ $(1\leq i\leq n)$ is obtained from $U(m)^{(n)}_0(=U(m)^{(n)}_{-1})$ by exchanging the dot and $0$ of the $i$-th $W_1$ component. Theorem~\ref{sec:algorithm1:main theorem} clearly holds for these $U(m)^{(n)}_i$ $(-1\leq i\leq n)$. In particular, $U(m)^{(n)}_i$ $(0\leq i\leq n)$ are mutually homeomorphic but not diffeomorphic, where we fix $n\geq 1$. \smallskip
\begin{figure}[ht!]
\begin{center}
\includegraphics[width=2.8in]{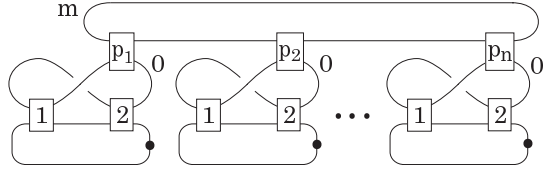}
\caption{A smooth handle diagram of $U(m)^{(n)}_0(=U(m)^{(n)}_{-1})$}
\label{fig17}
\end{center}
\end{figure}
\begin{figure}[ht!]
\begin{center}
\includegraphics[width=3.0in]{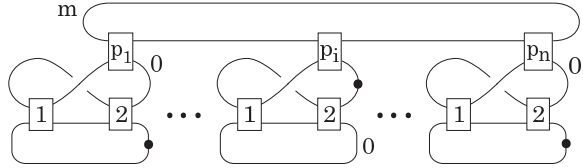}
\caption{A smooth handle diagram of $U(m)^{(n)}_i$ $(1\leq i\leq n)$}
\label{fig18}
\end{center}
\end{figure}

Finally, we discuss the (non-) existence of Stein structures on $U(m)^{(n)}_{-1}(=U(m)^{(n)}_{0})$.\smallskip

(1) The $m\leq -2$ case. In this case, $U(m)$ is a good Stein handlebody, hence Theorem~\ref{sec:algorithm1:main theorem} shows that $U(m)^{(n)}_{-1}$ admits a Stein structure. \smallskip

(2) The $m\geq -1$ case. In this case, $U(m)$ (and thus $U(m)^{(n)}_{-1}$) contains a homologically non-vanishing smoothly embedded sphere with its self-intersection number $m$. Adjunction inequality shows that $U(m)^{(n)}_{-1}$ does not admit any Stein structure. 


\subsection{Exotic complements in the 4-sphere}In this subsection, we demonstrate Theorem~\ref{sec:algorithm1:exotic knottings} in the $(Z,Y)=(S^4, \Sigma_g\times D^2)$ $(g\geq 1)$ case, where $\Sigma_g$ denotes the closed surface of genus $g$. Figure~\ref{fig19} is a handlebody diagram of $\Sigma_g\times D^2$.\smallskip
\begin{figure}[ht!]
\begin{center}
\includegraphics[width=3.0in]{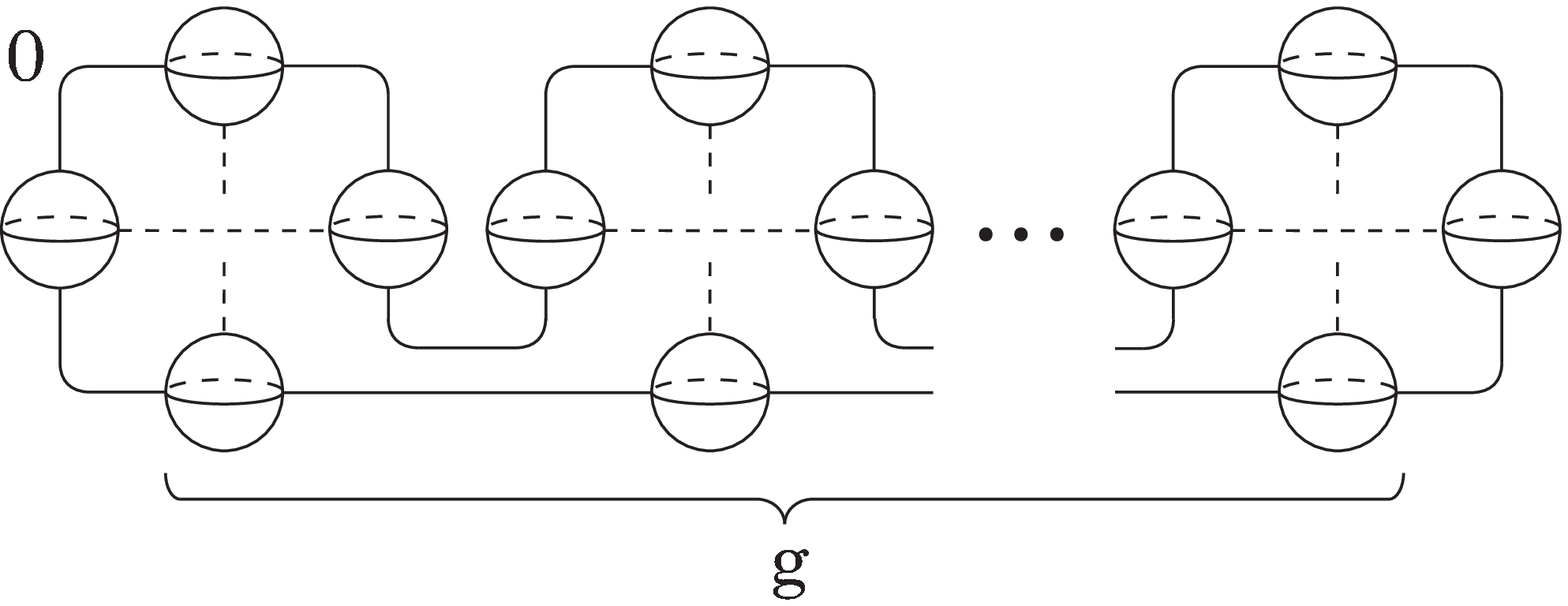}
\caption{$\Sigma_g\times D^2$ $(g\geq 1)$}
\label{fig19}
\end{center}
\end{figure}

We can embed $\Sigma_g\times D^2$ into $S^4$ as follows: Converting the picture into the dotted circle notation, we get the diagram of $\Sigma_g\times D^2$ in Figure~\ref{fig20}. Taking a double, we get the diagram of $\Sigma_g\times S^2$ in Figure~\ref{fig21}. By surgering $\Sigma_g\times D^2\subset \Sigma_g\times S^2$, we get the closed 4-manifold in Figure~\ref{fig22}, where the $0$-framed meridian of the dotted circle, $2g$ 3-handles and the $4$-handle constitute $\Sigma_g\times D^2$. It is easy to see that this closed 4-manifold is $S^4$ (cancel 1/2-handle pair, then cancel $2g$ 2/3-handle pairs).

\begin{figure}[ht!]
\begin{center}
\includegraphics[width=2.5in]{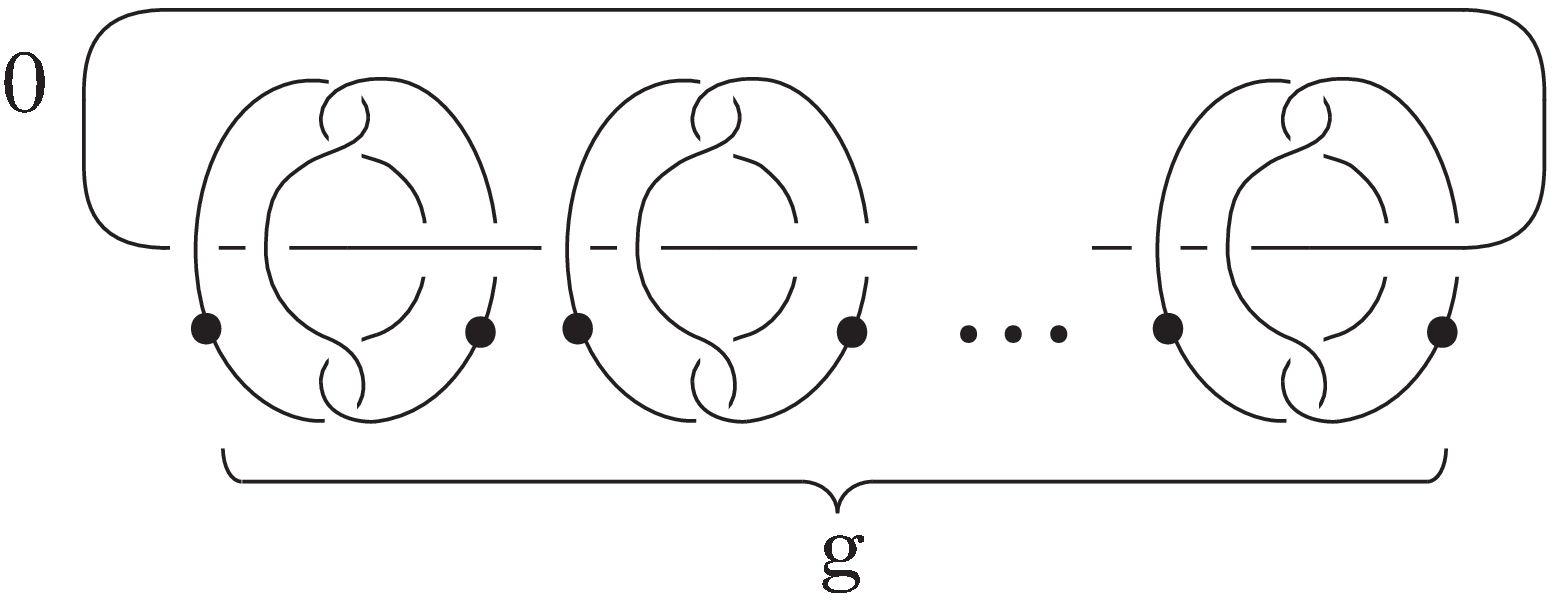}
\caption{$\Sigma_g\times D^2$ $(g\geq 1)$}
\label{fig20}
\end{center}
\end{figure}
\begin{figure}[ht!]
\begin{center}
\includegraphics[width=3.5in]{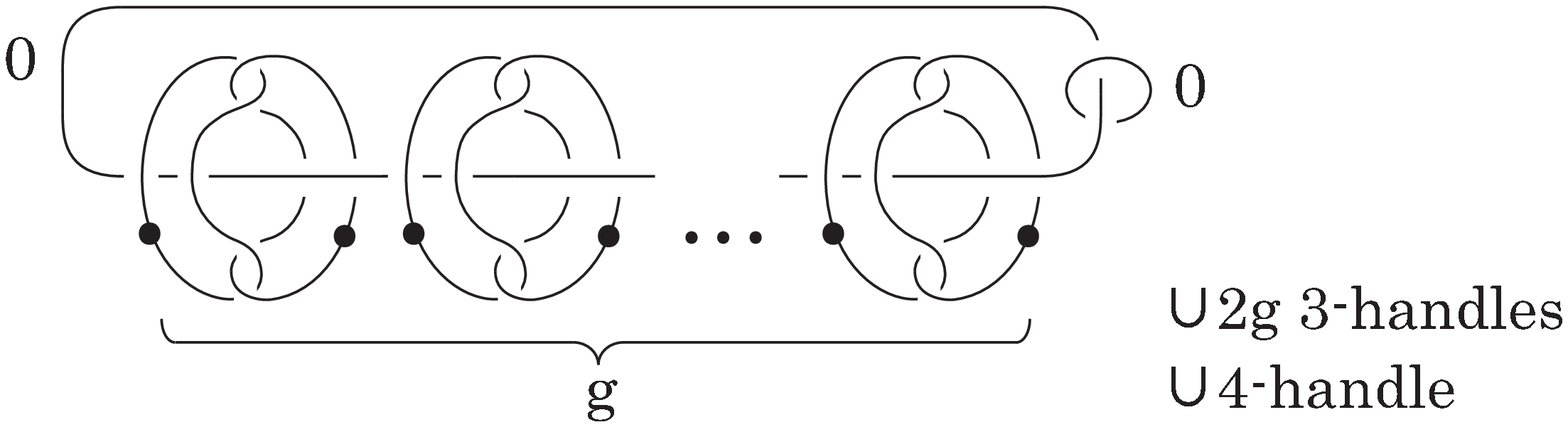}
\caption{$\Sigma_g\times S^2$ $(g\geq 1)$}
\label{fig21}
\end{center}
\end{figure}
\begin{figure}[ht!]
\begin{center}
\includegraphics[width=3.3in]{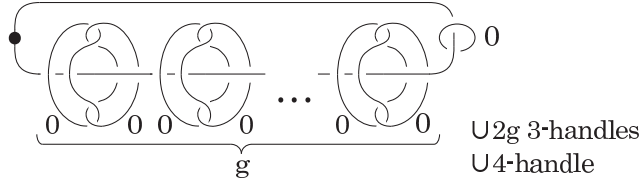}
\caption{$(\textnormal{surgered }\Sigma_g\times S^2)\cong S^4$ $(g\geq 1)$}
\label{fig22}
\end{center}
\end{figure}
\begin{figure}[ht!]
\begin{center}
\includegraphics[width=2.3in]{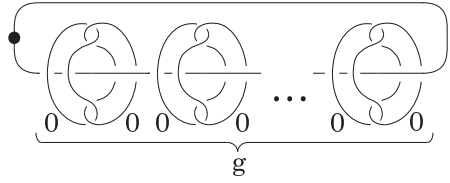}
\caption{$X_g:=S^4-\textnormal{int}\,(\Sigma_g\times D^2)$ $(g\geq 1)$}
\label{fig23}
\end{center}
\end{figure}
We thus have an embedding of $\Sigma_g\times D^2$ into $S^4$ such that its complement is the 2-handlebody (call $X_g$) with $b_2\geq 1$ in Figure~\ref{fig23}. \smallskip

Applying Theorem~\ref{sec:algorithm1:exotic knottings} to the above embedding, we get the following proposition. Note that the intersection form of $X_g$ is represented by the zero matrix.
\begin{proposition}\label{sec:ex:prop:complement}
Fix $g\geq 1$. For each $n\geq 1$, there exist mutually diffeomorphic compact connected oriented smooth $4$-manifolds $Y_i$ $(0\leq i\leq n)$ embedded into $S^4$ with the following properties. \smallskip\\
$(1)$ The pairs $(S^4, Y_i)$ $(0\leq i\leq n)$ are mutually homeomorphic but non-diffeomorphic.\smallskip\\
$(2)$ The fundamental group, the integral homology groups, the integral homology groups of the boundary, and the intersection form of every $Y_i$ $(0\leq i\leq n)$ are isomorphic to those of $\Sigma_g\times D^2$. 
\end{proposition}
\begin{remark} Similarly to the above, we can embed $D^2$-bundle $R_k$ $(k\geq 1)$ over $\#_{k}\mathbf{RP}^2$ with Euler number $-2k$ into $S^4$, and apply our algorithm to the pair $(S^4, R_k)$. While this procedure seems not to give exotic embeddings of $\#_{k}\mathbf{RP}^2$, Finashin-Kreck-Viro~\cite{FKV1}, \cite{FKV2} constructed infinitely many exotic embeddings of $\#_{10}\mathbf{RP}^2$ into $S^4$, by different methods. See also Finashin~\cite{Fina}. However, it is unclear whether the complements of (the neighborhoods of) their examples are mutually homeomorphic but not diffeomorphic. 
\end{remark}
\section{Exotic non-Stein 4-manifolds and exotic Stein 4-manifolds}\label{sec:non-stein} 
In this section, we construct arbitrary many non-Stein 4-manifolds and arbitrary many Stein 4-manifolds which are mutually homeomorphic but not diffeomorphic. Namely, we prove 
\begin{theorem}\label{sec:nonstein:thm:main0} Let $X$ be a $2$-handlebody with $b_2(X)\geq 1$, For each $n\geq 1$, there exist $2$-handlebodies $X^S_i$ and $X^N_i$ $(1\leq i\leq n)$ with the following properties.\smallskip\\
$(1)$ $X^S_1, X^S_2,\dots, X^S_n$ and $X^N_1, X^N_2,\dots,X^N_n$ are mutually homeomorphic but non-diffeomorphic. \smallskip\\
$(2)$ Every $X^S_i$ $(1\leq i\leq n)$ admits a Stein structure, and any $X^N_i$ $(1\leq i\leq n)$ admits no Stein structure. \smallskip\\
$(3)$ The fundamental groups, the integral homology groups, the integral homology groups of the boundary, and the intersection forms of every $X^S_i$ and every $X^N_i$ $(0\leq i\leq n)$ are isomorphic to those of the boundary sum $X\natural U(0)$. Here $U(0)$ denotes the one in Subsection~\ref{subsec:simplestexample}. 
\end{theorem}
We prove this theorem, using the examples $U(0)^{(1)}_0$ and $U(0)^{(1)}_1$ in subsection~\ref{subsec:simplestexample}. Figure~\ref{fig24} shows smooth handlebody diagrams of $U(0)^{(1)}_0$ and $U(0)^{(1)}_1$. 
\begin{figure}[ht!]
\begin{center}
\includegraphics[width=2.6in]{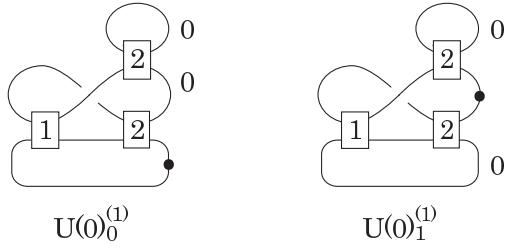}
\caption{$U(0)^{(1)}_0$ and $U(0)^{(1)}_1$}
\label{fig24}
\end{center}
\end{figure}

\begin{definition}\label{sec:nonstein:def:nonstein} 
Let $X$ be a $2$-handlebody with $b_2(X)\geq 1$. Fix $n\geq 1$. Let $\widehat{X}^{(n)}_i$ $(1\leq i\leq n)$ be the one in  Definition~\ref{sec:variant:def}, corresponding to this $X$. We assume that $p_i$'s in Definition~\ref{sec:variant:def} further satisfy the following condition. 
\begin{itemize}
 \item [(iii)'] $2p_1+(t_0-1)-m_0+\rvert r_0\lvert \;> 2$.
\end{itemize}
For each $1\leq i\leq n$, put $X^S_i:=\widehat{X}^{(n)}_i\natural U(0)^{(1)}_1$ and $X^N_{i}:=\widehat{X}^{(n)}_i\natural U(0)^{(1)}_0$ where $\natural$ denotes the boundary sum.
\end{definition}

\begin{lemma}\label{sec:nonstein:lem:stein_nonstein}
$(1)$ Every $X^S_i$ $(1\leq i\leq n)$ admits a Stein structure. \smallskip\\
$(2)$ Any $X^N_i$ $(1\leq i\leq n)$ admits no Stein structure for any orientations.\smallskip\\
$(3)$ $X^S_1, X^S_2,\dots, X^S_n$ and $X^N_1, X^N_2,\dots,X^N_n$ are mutually homeomorphic.\smallskip\\ 
$(4)$ The fundamental groups, the integral homology groups, the integral homology groups of the boundary, and the intersection forms of every $X^S_i$ and every $X^N_i$ $(0\leq i\leq n)$ are isomorphic to those of $X\natural U(0)$. 
\end{lemma}
\begin{proof}Every $X^S_i$ $(1\leq i\leq n)$ is clearly a Stein handlebody, hence admits a Stein structure. Since $U(0)^{(1)}_0$ (hence $X^N_i$) contains a homologically non-vanishing smoothly embedded sphere with its self-intersection number $0$, the adjunction inequality shows that $X^N_j$ $(1\leq j\leq n)$ does not admit any Stein structure for any orientations. Since $X^S_1, X^S_2,\dots, X^S_n$ and $X^N_1, X^N_2,\dots,X^N_n$ are mutually related by cork twists, the claim $(3)$ follows. The definitions of $U(0)^{(1)}_j$, $X^S_i$ and $X^N_i$ and Theorem~\ref{sec:variant:main theorem} show the claim $(4)$.
\end{proof}
\begin{lemma}\label{sec:nonstein:lem:smooth:stein} $X^S_i$ $(1\leq i\leq n)$ are mutually non-diffeomorphic for any orientations. 
\end{lemma}
\begin{proof}The constructions in Sections~\ref{sec:main_algorithm} and \ref{sec:variant} and the above condition (iii)' show that each $X^S_i$ can be diffeomorphic to $\widehat{X\natural U(0)}^{(n)}_{i}$ which denotes the Legendrian handlebody in Theorem~\ref{sec:variant:main theorem} corresponding to $X\natural U(0)$. Theorem~\ref{sec:variant:main theorem} thus shows that $X^S_i$ $(1\leq i\leq n)$ are mutually non-diffeomorphic for any orientations.
\end{proof}

Let $v^{(i)}_j$ $(0\leq j\leq k)$ denote the elements of $H_2(\widehat{X}^{(n)}_{i};\mathbf{Z})$ in Definition~\ref{sec:algorithm:def:v}, corresponding to $\widehat{X}^{(n)}_{i}$ in Definition~\ref{sec:nonstein:def:nonstein} (recall that $\widehat{X}^{(n)}_{i}$ is defined as a special case of $X^{(n)}_{i}$). Let $w$ be the generator of $H_2(U(0)^{(1)}_{0};\mathbf{Z})$. Let $g_j$ $(0\leq j\leq k)$ and $m_j, q_j$ $(0\leq j\leq l)$ denote the integers in Section~\ref{sec:main_algorithm}, corresponding to the above $\widehat{X}^{(n)}_{i}$. Then the following lemma holds similarly to Lemma~\ref{sec:main:lem:genus}. 
 
\begin{lemma}\label{sec:nonstein:lem:genus1} For each $0\leq i\leq n$, the elements $v^{(i)}_0,v^{(i)}_1,\dots,v^{(i)}_k,w$ span a basis of $H_2(X^N_i;\mathbf{Z})$ and satisfy the following conditions $($ignore \textnormal{(ii)} when $k=0$$)$.
\smallskip
\begin{itemize}
  \item [(i)] $v^{(i)}_0$ is represented by a smoothly embedded surface of genus $g_0+p_i$ and satisfies $v^{(i)}_0\cdot v^{(i)}_0=m_0$.\smallskip 
 \item [(ii)] $v^{(i)}_j$ $(1\leq j\leq k)$ is represented by a smoothly embedded surface of genus $g_j+q_j$ and satisfies $v^{(i)}_j\cdot v^{(i)}_j=m_j$.\smallskip 
 \item [(iii)] $w$ is represented by a smoothly embedded sphere and satisfies $w^2=0$.
\end{itemize}
\end{lemma}\smallskip

We can easily check the following lemma similarly to Propositions~\ref{sec:main theorem:prop:genus} and \ref{sec:variant:prop:non-existence}. 

\begin{lemma}\label{sec:nonstein:lem:genus2}  For any $1\leq i\leq n$, there exists no basis $u_0,u_1,\dots,u_{k+1}$ of $H_2(X^N_i;\mathbf{Z})$ which satisfies the following conditions $($ignore \textnormal{(ii)} when $b_2(X)=1$$)$.\smallskip
\begin{itemize}
 \item [(i)] $u_0$ is represented by a smoothly embedded surface with its genus equal to or less than $g_0+p_{i-1}$ and satisfies $u^2_0=\lvert m_0\rvert$.\smallskip
 \item [(ii)] Each $u_j$ $(1\leq j\leq k)$ is represented by a smoothly embedded surface of genus $g_j+q_j$ and satisfies $u^2_j=\lvert m_j\rvert$.\smallskip
 \item [(iii)] $u_{k+1}$ is represented by a smoothly embedded sphere and satisfies $u_{k+1}^2=0$.
\end{itemize}
\end{lemma}
\begin{proof} Suppose that such a basis $u_0,u_1,\dots,u_{k+1}$ exists. Then each $u_j$ $(0\leq j\leq k+1)$ is a linear combination of $v^{(i)}_0,v^{(i)}_1,\dots,v^{(i)}_k,w$ by Lemma~\ref{sec:nonstein:lem:genus1}. Since $U(0)$ can be embedded into the 4-ball, Proposition~\ref{sec:attachment:prop:embed} shows that $U(0)^{(1)}_{0}$ can be embedded into the 4-ball. Thus, each $X^N_i$ can be embedded into $\widehat{X}^{(n)}_{i}$ so that $w$ is sent to $0$ and that $v^{(i)}_0,v^{(i)}_1,\dots,v^{(i)}_k$ are sent identically. Therefore we can apply the same argument as Propositions~\ref{sec:main theorem:prop:genus} and \ref{sec:variant:prop:non-existence}, and easily get the required claim. 
\end{proof}
We can now easily prove Theorem~\ref{sec:nonstein:thm:main0}. 
\begin{proof}[Proof of Theorem~\ref{sec:nonstein:thm:main0}]
Lemma~\ref{sec:nonstein:lem:stein_nonstein} gives the claims $(2)$ and $(3)$. Lemma~\ref{sec:nonstein:lem:stein_nonstein} also shows that, for any $i, j$, two 4-manifolds $X^S_i$ and $X^N_j$ are not diffeomorphic for any orientations. Since the sequence $p_i$ $(i\geq 0)$ is strictly increasing, Lemmas~\ref{sec:nonstein:lem:genus1} and \ref{sec:nonstein:lem:genus2} show that $X^N_i$ $(1\leq i\leq n)$ are mutually non-diffeomorphic for any orientations. Lemmas~\ref{sec:nonstein:lem:smooth:stein} and \ref{sec:nonstein:lem:stein_nonstein}.(3) thus give the claim $(1)$. 
\end{proof}
\begin{remark}Though we used $U(0)^{(1)}_{j}$ to define $X^S_i$ and $X^N_i$, we can similarly define $X^S_i$ and $X^N_i$, using $U(-1)^{(1)}_{j}$. Put $X^S_i:=\widehat{X}^{(n)}_i\natural U(-1)^{(1)}_1$ and $X^N_i:=\widehat{X}^{(n)}_i\natural U(-1)^{(1)}_0$, where we assume that $p_i$'s satisfy
\begin{itemize}
 \item [(iii)''] $2p_1+(t_0-1)-m_0+\rvert r_0\lvert -1> 0$, 
\end{itemize}
instead of (iii)' in Definition~\ref{sec:nonstein:def:nonstein}. In this case, Theorem~\ref{sec:nonstein:thm:main0} also holds, where we replace $X\natural U(0)$ with $X\natural U(-1)$ in the claim $(3)$. Moreover, each $X^N_i$ does admit a Stein structure after blowing down. \smallskip

We here give an outline of this proof. The claim corresponding to Lemmas~\ref{sec:nonstein:lem:stein_nonstein}, \ref{sec:nonstein:lem:smooth:stein} and \ref{sec:nonstein:lem:genus1} clearly holds. However, the claim corresponding to Lemma~\ref{sec:nonstein:lem:genus2} is not clear, because $U(-1)$ cannot be embedded into the 4-ball. Here notice that $U(-1)^{(1)}_{0}$ contains a 2-sphere with the self-intersection number $-1$, and that the blowdown of $U(-1)^{(1)}_{0}$ still has a Stein handlebody presentation. Using this fact, we can prove the claim corresponding to Lemma~\ref{sec:nonstein:lem:genus2} as follows. Since the blowdown of $X^N_i$ is a Stein handlebody, we can embed it into a minimal complex surface of general type (for this, the property similar to Theorem~\ref{sec:stein:th:stein_rotation} holds. cf.~\cite{GS}.). Then use the blow up formula and the adjunction inequality, and apply the argument in the proof of Lemma~\ref{sec:nonstein:lem:genus2}.
\end{remark}
$2$-handlebodies give a large class of $4$-manifolds with boundary. Actually, we easily get the following. 

\begin{corollary}\label{sec:non-stein:cor}$(1)$ For any finitely presented group $G$, there exist arbitrary many compact Stein $4$-manifolds and arbitrary many non-Stein $4$-manifolds such that they are mutually homeomorphic but not diffeomorphic and that their fundamental groups are isomorphic to $G$.\smallskip\\
$(2)$ For any integral symmetric bilinear form $Q$ over any integral free module, there exist arbitrary many simply connected compact Stein $4$-manifolds such that they are mutually homeomorphic but not diffeomorphic and that their intersection forms are isomorphic to $Q$.
\end{corollary}


\end{document}